\documentclass[12pt,preprint]{amsart}
\usepackage{amsthm,amsmath,amssymb}
\usepackage{amssymb,amsfonts,amsthm,amsmath,amscd}
\usepackage{graphics}
\usepackage{graphicx}
\usepackage{color}
\usepackage{supertabular}
\usepackage{tikz}
\usepackage{multirow}
\usepackage{enumerate}

\newtheorem{theorem}{Theorem}
\newtheorem{lemma}{Lemma}
\newtheorem{proposition}{Proposition}
\newtheorem{corollary}{Corollary}

\theoremstyle{definition}

\newtheorem{remark}{Remark}
\newtheorem{definition}{Definition}

\def\G{\mathcal{G}}
\def\P{\mathcal{P}}
\def\N{\mathcal{N}}
\def\S{\mathcal{S}}
\def\Z{\mathbb{Z}}

\newcommand{\cH}{{\mathcal{H}}}

\begin{document}

\title{On tame, pet, domestic, and miserable impartial games}

\author{Vladimir Gurvich}
\address{MSIS and RUTCOR, RBS, Rutgers University,
100 Rockafeller Road, Piscataway, NJ 08854;
National Research University Higher School of Economics (HSE), Moscow}
\email{vladimir.gurvich@rutgers.edu, vladimir.gurvich@gmail.com}

\author{Nhan Bao Ho}
\address{Department of Mathematics and Statistics, La Trobe University, Melbourne, Australia 3086}
\email{nhan.ho@latrobe.edu.au, nhanbaoho@gmail.com}

\subjclass[2000]{Primary: 91A46}
\keywords{Moore's $k$-{\sc Nim}, extended complementary {\sc Nim}, impartial games, $\P$-positions, Sprague-Grundy function}

\thanks{The first author gratefully acknowledge partial support
of the Russian Academic Excellence Project `5-100'.}

\begin{abstract}
Playing impartial games under the normal and mis\`{e}re conventions may differ a lot.
However, there are also many ``exceptions" for which
the normal and mis\`{e}re plays are very similar.
As early as in 1901 Bouton noticed that this is the case with the game of {\sc Nim}.
In 1976 Conway introduced a large class of such games that he called {\em tame} games.
Here we introduce a proper subclass, {\em pet} games, and
a proper superclass, {\em domestic} games.
For each of these three classes we provide an efficiently verifiable characterization
based on the following property.
These games are closely related to another important subclass of the tame games
introduced in 2007 by the first author and called {\em miserable} games.
We show that tame, pet, and domestic games turn into miserable games
by ``slight modifications" of their definitions.
We also show that the sum of miserable games is miserable and
find several other classes that respect summation.
The developed techniques allow us to prove that
very many well-known impartial games fall into classes mentioned above.
Such examples include all subtraction games, which are pet;
game {\sc Euclid}, which is miserable (and, hence, tame),
as well as many versions of the {\sc Wythoff} game and {\sc Nim}, which
may be miserable, pet, or domestic.
\end{abstract}

\maketitle


\section{Sprague-Grundy theory of impartial games}
\label{S.Game}

Combinatorial games were analyzed in the comprehensive books \cite{BCG01-04} and
earlier in \cite{Con76}; an introductory theory can be found in \cite{ANW07, CGT}.
Readers familiar with the subject can skip this section.
We restrict ourselves to a special case. A game is called
\begin{itemize}
\item{} {\em impartial} if both players have the same possible moves in each position;
\item{} {\em acyclic} if each position can be visited at most once;
\item{} {\em finite} if the set of positions is finite;
\item{} {\em locally finite} if the subgame defined by any fixed initial position is finite.
\end{itemize}

In this paper we consider only the locally finite acyclic impartial games of two players,
calling them simply {\it games}, for brevity.
We say that a game is played under the {\em normal} (resp.,~{\em mis\`ere}) convention
if the player who makes the last move wins (resp.,~loses). We will consider both.

\medskip

Games are modeled by finite acyclic digraphs
whose vertices are interpreted as {\em positions} and arcs as {\em moves}.
In case there exists a move  $(x, y)$ from a position  $x$ to $y$, we write
$x \rightarrow y$ and say that  $x$ is {\it movable} to  $y$, or
$y$ is {\it reachable} from  $x$, or  $y$ is an {\it option} of $x$.

Similarly, given two sets of positions $U, W \subseteq V$,  we say that
$U$  is {\it movable} to  $W$  or that
$W$  is {\it reachable} from  $U$,  if
from every position $x \in U$  there is a move to some position $y \in W$.

A position without available moves is called {\em terminal}. The set of terminal positions is denoted by  $V_T$. A position is called an $\N$-position (resp.,~a $\P$-position) if the next (resp.,~previous) player wins when both players play optimally starting from that position.

Given a set  $S$ of non-negative integers,
the {\it minimum excludant}  of  $S$, denoted by $\operatorname{mex}(S)$,
defined as the least non-negative integer that is not in $S$.
In particular,  $\operatorname{mex}(S) = 0$  whenever  $0 \not \in S$,
for example, if  $S = \emptyset$.

The {\it Sprague-Grundy (SG) function} of a game $G$,
denoted by $\G$,  is defined recursively as follows
\begin{align} \label{R.SG}
\G(x) = mex\{\G(y) \mid \text{ $y$  is an option of  $x$}\}.
\end{align}
The value  $\G(x)$  is called the {\it SG value},
or alternatively the {\it nim-value}, of position $x$.
By the above definition, $\G(x) = 0$ whenever  $x$  is a terminal position.
It is both obvious and well-known \cite{Spr35, Spr37, Gru39} that
the SG values are characterized as follow.

\begin{lemma}\label{SG}
We have  $\G(x) = n$  if and only if the next two conditions hold:
\begin{enumerate} \itemsep0em
\item [\rm{(i)}] $\G(y) \neq \G(x)$  whenever there is a move from  $x$ to $y$, in particular, $\G(y) \neq n$  if  $\G(x)= n$;
\item [\rm{(ii)}] for each integer $k$  such that $0 \leq k < n$ there exists a move $x \rightarrow y$  such that  $\G(y) = k$.
\end{enumerate}
In particular, the $\P$-positions are exactly the zeros of the SG function.
\qed
\end{lemma}

Given two games $G$ and $H$, their disjunctive sum $G+H$ is defined as
a game in which every move consists of choosing one game and making a move in it.
The SG function of the sum is characterized by the following well-known statement.
Let $\oplus$  be the bitwise addition in the binary number system without carrying,
or in other words, the bitwise $\bmod 2$ addition.

\begin{theorem} \cite{Gru39, Spr35, Spr37}
\label{S.SG}
The SG value of the position $(x, y)$  in the sum $G + H$  is $\G(x) \oplus \G(y)$.
\qed
\end{theorem}

This result can be obviously extended to the sums of  $k$  games for any integer  $k \geq 2$, since
$\oplus$  is an associative and commutative operation.

\medskip

The mis\`{e}re SG value $\G^-(x)$  of a position $x$ in a game  $G$
is defined by the same recursion (\ref{R.SG}), but the initialization is different:
$\G^-(x) = 1$  (rather than $\G^-(x) = 0$) for all terminal positions $x \in V_T$.

\begin{remark}
\label{normal-misere}
For an individual game, the mis\`{e}re version can be easily reduced to the normal one
by the following simple transformation of the graph  $G = (V,E)$.
Add to  $V$  one new position  $x_T$  and an arc
$(x,x_T)$ from each former terminal position  $x \in V_T$ to $x_T$.
Thus, $x_T$ becomes a unique terminal in the obtained graph  $G^-$.
It is easy to verify that for every position  $x \in V$
its mis\`{e}re SG value  $\G^-(x)$,
in the original digraph  $G$, equals the normal SG value $\G(x)$ the extended digraph $G^-$.

However, then the mis\`{e}re version of a sum 
and the sum of the mis\`{e}re versions of the summands
are not the same.
In the first case we add only one new
terminal position for the whole sum, while
in the second case we have to add one for each game summand.

As early as in 1956 Grundy and Smith  \cite{GS56} noticed that
playing a game under the mis\`{e}re convention may be difficult in general.
In this paper we focus on the exceptions, that is, on the games
for which functions  $\G$ and $\G^-$ are closely related.
 \end{remark}


\section{Main concepts and results}
\label{S.Intro}

A position $x$  will be called an {\em $i$-position}
(resp.,~an $(i,j)$-position) if  $\G(x) = i$
(resp.,~if  $\G(x) = i$  and  $\G^-(x) = j$).
We will denote by $V_i$
(resp.,~$V_{i,j}$) the set of $i$-positions (resp.,~$(i,j)$-positions).
A position $x \in V_{0,1} \cup V_{1,0}$  will be called a {\em swap position}.

\begin{definition} \label{D.DTP}
An impartial game will be called
\begin{enumerate} [(i)] \itemsep0em
\item {\em domestic} if it has neither $(0,k)$-positions nor $(k,0)$-positions with $k \geq 2$;
\item {\em tame} if it has only $(0,1)$-positions, $(1,0)$-positions, and $(k,k)$-positions with $k \geq 0$;
\item {\em pet} if it has only $(0,1)$-positions, $(1,0)$-positions, and $(k,k)$-positions with $k \geq 2$.
\end{enumerate}
\end{definition}

Tame games were introduced in \cite[Chapter 12]{Con76}
(see page 178), pet games were introduced recently in the preprints \cite{Gur11, Gur120},
while domestic games are introduced in this paper.
According to the above definitions, domestic, tame, and pet games form nested classes:
a pet game is tame and a tame game is domestic.
Furthermore, both containments are strict.
Figures \ref{nD}, \ref{DnT}, \ref{TnP}, and \ref{Pet} distinguish these three classes.

\begin{figure}[ht]
\begin{center}
\begin{tikzpicture}[style={draw=blue,thick,circle,inner sep=0pt}] 
 \node at (0in,0in) [shape=circle,minimum size=.6cm,draw=black,thick] (A) {0,1};
 \node at (.75in,0in) [shape=circle,minimum size=5pt,draw=black,thick] (B) {1,0};
 \node at (1.5in,0in) [shape=circle,minimum size=5pt,draw=black,thick] (C) {2,2};
 \node at (2.25in,0in) [shape=circle,minimum size=5pt,draw=black,thick] (D) {0,0};
 \node at (3in,0in) [shape=circle,minimum size=5pt,draw=black,thick] (E) {1,1};
 \node at (3.75in,0in) [shape=circle,minimum size=5pt,draw=black,thick] (F) {2,0};
 \node at (4.5in,0in) [shape=circle,minimum size=5pt,draw=black,thick] (G) {3,2};
\draw [->] (B) to (A);
\draw [->] (C) to (B);
\draw [->] (C) to [out=155,in=25] (A);
\draw [->] (D) to (C);
\draw [->] (E) to (D);
\draw [->] (F) to (E);
\draw [->] (F) to [out=150,in=30] (A);
\draw [->] (G) to (F);
\draw [->] (G) to [out=160,in=20] (E);
\draw [->] (G) to [out=155,in=25] (D);
\end{tikzpicture}
\end{center}
\caption{\label{nD}
This game is not domestic since it contains a (2,0)-position.}
\end{figure}
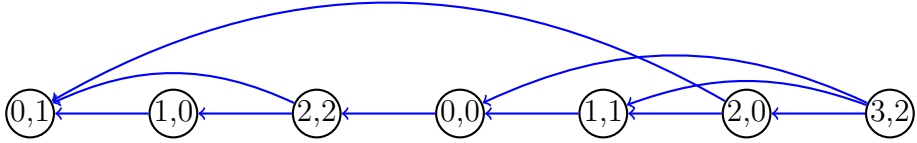

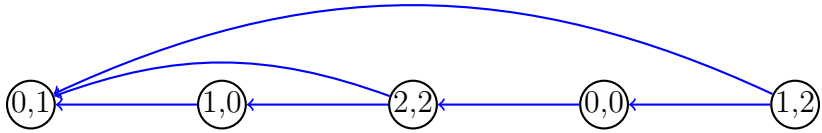
\begin{figure}[ht]
\begin{center}
\begin{tikzpicture}[style={draw=blue,thick,circle,inner sep=0pt}]   
 \node at (0in,0in) [shape=circle,minimum size=.6cm,draw=black,thick] (A) {0,1};
 \node at (1in,0in) [shape=circle,minimum size=10pt,draw=black,thick] (B) {1,0};
 \node at (2in,0in) [shape=circle,minimum size=10pt,draw=black,thick] (C) {2,2};
 \node at (3in,0in) [shape=circle,minimum size=10pt,draw=black,thick] (D) {0,0};
 \node at (4in,0in) [shape=circle,minimum size=10pt,draw=black,thick] (E) {1,2};
\draw [->] (D) to (C);
\draw [->] (C) to (B);
\draw [->] (B) to (A);
\draw [->] (C) to [out=160,in=20] (A);
\draw [->] (E) to (D);
\draw [->] (E) to [out=155,in=25] (A);
\end{tikzpicture}
\end{center}
\caption{\label{DnT}
This game is domestic but not tame since it contains a (1,2)-position.}
\end{figure}

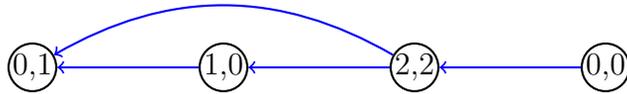
\begin{figure}[ht]
\begin{center}
\begin{tikzpicture}[style={draw=blue,thick,circle,inner sep=0pt}] 
 \node at (0in,0in) [shape=circle,minimum size=.6cm,draw=black,thick] (A) {0,1};
 \node at (1in,0in) [shape=circle,minimum size=10pt,draw=black,thick] (B) {1,0};
 \node at (2in,0in) [shape=circle,minimum size=10pt,draw=black,thick] (C) {2,2};
 \node at (3in,0in) [shape=circle,minimum size=10pt,draw=black,thick] (D) {0,0};
\draw [->] (D) to (C);
\draw [->] (C) to (B);
\draw [->] (B) to (A);
\draw [->] (C) to [out=150,in=30] (A);
\end{tikzpicture}
\end{center}
\caption{\label{TnP}   
This game is tame but not pet since it contains a (0,0)-position;
also it is miserable but not strongly miserable.}
\end{figure}

\begin{figure}[ht]
\begin{center}
\begin{tikzpicture}[style={draw=blue,thick,circle,inner sep=0pt}] 
 \node at (0in,0in) [shape=circle,minimum size=.6cm,draw=black,thick] (A) {0,1};
 \node at (1in,0in) [shape=circle,minimum size=10pt,draw=black,thick] (B) {1,0};
 \node at (2in,0in) [shape=circle,minimum size=10pt,draw=black,thick] (C) {2,2};
\draw [->] (C) to (B);
\draw [->] (B) to (A);
\draw [->] (C) to [out=150,in=30] (A);
\end{tikzpicture}
\end{center}
\caption{\label{Pet}
This game is pet.}
\end{figure}
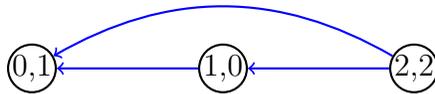

The following two ``technical" properties
appear to be closely related to the above three classes of games.

\begin{definition} \label{D.FR} \label{d1}
A game  $G$  is said to be
\begin{enumerate} [(i)] \itemsep0em
\item {\em forced} if each move from a $(0,1)$-position results in a $(1,0)$-position and vice versa;
\item {\em returnable} if the following, weaker, implications hold:
let $x$ be a $(0,1)$-position (resp.,~a $(1,0)$-position) movable to a non-terminal position $y$, then
$y$ is movable to a $(0,1)$-position (resp.,~to a $(1,0)$-position).
\end{enumerate}
\end{definition}

Obviously, the forced games are returnable.
Figures \ref{nR} and  \ref{RnF}  give examples of a non-returnable game and
a returnable game that is not forced, respectively.
Note that the games in Figures \ref{nD}, \ref{DnT}, \ref{TnP}, and \ref{Pet} are all forced.

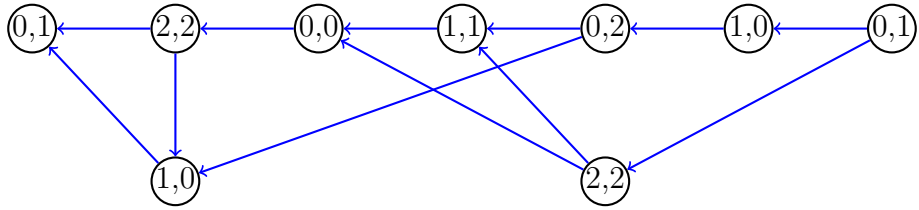
\begin{figure}[ht]
\begin{center}
\begin{tikzpicture}[style={draw=blue,thick,circle,inner sep=0pt}] 
 \node at (0in,0in) [shape=circle,minimum size=.6cm,draw=black,thick] (A) {0,1};
 \node at (.75in,-.8in) [shape=circle,minimum size=5pt,draw=black,thick] (B) {1,0};
 \node at (.75in,0in) [shape=circle,minimum size=5pt,draw=black,thick] (C) {2,2};
 \node at (1.5in,0in) [shape=circle,minimum size=5pt,draw=black,thick] (D) {0,0};
 \node at (2.25in,0in) [shape=circle,minimum size=5pt,draw=black,thick] (E) {1,1};
 \node at (3in,0in) [shape=circle,minimum size=5pt,draw=black,thick] (F) {0,2};
 \node at (3.75in,0in) [shape=circle,minimum size=5pt,draw=black,thick] (G) {1,0};
 \node at (3in,-.8in) [shape=circle,minimum size=5pt,draw=black,thick] (H) {2,2};
 \node at (4.5in,0in) [shape=circle,minimum size=5pt,draw=black,thick] (I) {0,1};
\draw [->] (D) to (C);
\draw [->] (C) to (B);
\draw [->] (B) to (A);
\draw [->] (C) to (A);
\draw [->] (E) to (D);
\draw [->] (F) to (E);
\draw [->] (F) to (B);
\draw [->] (G) to (F);
\draw [->] (H) to (E);
\draw [->] (H) to (D);
\draw [->] (I) to (H);
\draw [->] (I) to (G);
\end{tikzpicture}
\end{center}
\caption{\label{nR}
This game is not returnable.}
\end{figure}

\begin{figure}[ht]
\begin{center}
\begin{tikzpicture}[style={draw=blue,thick,circle,inner sep=0pt}] 
 \node at (0in,0in) [shape=circle,minimum size=.6cm,draw=black,thick] (A) {0,1};
 \node at (.75in,0in) [shape=circle,minimum size=5pt,draw=black,thick] (B) {1,0};
 \node at (1.5in,0in) [shape=circle,minimum size=5pt,draw=black,thick] (C) {2,2};
 \node at (2.25in,0in) [shape=circle,minimum size=5pt,draw=black,thick] (D) {0,0};
 \node at (3in,0in) [shape=circle,minimum size=5pt,draw=black,thick] (E) {1,1};
 \node at (3.75in,0in) [shape=circle,minimum size=5pt,draw=black,thick] (F) {2,0};
 \node at (4.5in,0in) [shape=circle,minimum size=5pt,draw=black,thick] (G) {3,2};
 \node at (4.125in,.6in) [shape=circle,minimum size=5pt,draw=black,thick] (H) {0,1};
\draw [->] (B) to (A);
\draw [->] (C) to (B);
\draw [->] (C) to [out=155,in=25] (A);
\draw [->] (D) to (C);
\draw [->] (E) to (D);
\draw [->] (F) to (E);
\draw [->] (F) to [out=150,in=30] (A);
\draw [->] (G) to (F);
\draw [->] (G) to [out=160,in=20] (E);
\draw [->] (G) to [out=150,in=30] (D);
\draw [->] (G) to (H);
\draw [->] (H) to (F);
\end{tikzpicture}
\end{center}
\caption{\label{RnF}
This game is returnable but not forced.}
\end{figure}
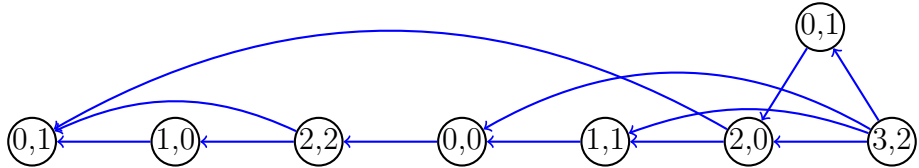

\begin{definition}
\label{D.WTMS}
For each position  $x$  let us consider the following properties:
\begin{itemize} \itemsep0em
\item [$(\mathfrak{a})$]   $x$ is a swap position, $x \in V_{0,1} \cup V_{1,0}$;
\item [$(\mathfrak{a_0})$] $x \in V_{0,1} \cup V_{1,0} \cup V_{0,0} \cup V_{1,1}$;
\item [$(\mathfrak{b})$]   $x$ is not movable to $V_{0,1} \cup V_{1,0}$;
\item [$(\mathfrak{c})$]   $x$ is movable to $V_{0,1}$ and to $V_{1,0}$ simultaneously;
\item [$(\mathfrak{c_0})$] $x$ is movable to $V_{0,1}$ and to $V_{0,0}$ simultaneously;
\item [$(\mathfrak{c_1})$] $x$ is movable to $V_{1,0}$ and to $V_{0,0}$ simultaneously;
\item [$(\mathfrak{e})$]   $x$ is movable to $V_{0,0}$ and to $V_{1,1}$ simultaneously.
\end{itemize}

A game is called
\begin{enumerate} [(i)] \itemsep0em
\item  {\em strongly miserable} if either $(\mathfrak{a})$ or $(\mathfrak{c})$ hold for every position;
\item  {\em miserable} if $(\mathfrak{a})$, or $(\mathfrak{b})$, or $(\mathfrak{c})$ hold for every position;
\item  {\em $t$-miserable} if ($\mathfrak{a_0}$), 
    or $(\mathfrak{c})$, or $(\mathfrak{e})$ hold for every position;
\item  {\em weakly miserable} if  $(\mathfrak{a})$, or  $(\mathfrak{b})$, or  $(\mathfrak{c})$, or  $(\mathfrak{c_0})$, or $(\mathfrak{c_1})$ hold for every position.
\end{enumerate}
\end{definition}


The classes of miserable and strongly miserable games were introduced in
\cite{Gur07} and \cite{Gur11, Gur120}, respectively.
It is not difficult to verify that four classes of Definition \ref{D.WTMS} are nested.
Furthermore, three examples in Figures \ref{DnT}, \ref{TnP}, and \ref{TnM}  show
that the containments are strict.


\medskip

The paper is organized as follows.
In Section  \ref{S.Nested} we show that a game is domestic, tame, or pet if and only if
it is weakly miserable, $t$-miserable, or strongly miserable, respectively.

Let us note, however, that these effective characterizations in terms of ``miserability"
still do not provide efficient membership tests,
because verifying properties of Definition \ref{D.WTMS}
requires knowledge of sets $V_{0,1}$, $V_{1,0}$, $V_{0,0}$, and $V_{1,1}$
that are defined recursively.
In Section \ref{S.Equi} we reformulate ``slightly" these properties
to obtain an efficient way to verify the membership in the classes
of domestic, tame, and pet games.

\medskip

We say that a class of games is preserved under summation
if the sum of games from this class also belongs to it.
In Section \ref{S.Sum}  we prove that the classes of
tame games, miserable games, forced and miserable games, returnable and miserable games
are preserved under summation, while pet and domestic games are not.
For tame games the result was stated in \cite{Con76} and proved in \cite{CGT};
we provide a simpler proof.

\smallskip

In Section \ref{S.App} we apply the results of Section \ref{S.Equi}
for several well-known classes of games, including
{\sc Subtraction games}, {\sc Euclid}, {\sc Nim}, {\sc Wythoff},
as well as for several modifications and generalizations of these games.


\section{Containment and equalities}
\label{S.Nested}

\subsection{Summary}


The following classes of games are shown to be identical:
\begin{itemize} \itemsep0em
\item domestic games and weakly miserable games;
\item tame games and $t$-miserable games;
\item pet games and strongly miserable games.
\end{itemize}
Furthermore, the following strict containments hold:

\begin{itemize} \itemsep0em
\item the pet (strongly miserable) games are miserable and the latter are tame.
\end{itemize}

We illustrate relations between the six considered classes by the diagram in Figure \ref{Inc}.

\begin{figure}[ht]
\begin{center}
\setlength{\unitlength}{1cm}
\begin{picture}(8, 5)(-3,.6)
\put(.5,2.8){\color{red}\oval(8,  1)}
\put(.5,2.6){\color{green}\oval(9,  1.7)}
\put(.5,2.4){\color{blue}\oval(10,  2.4)}
\put(.5,2.2){\color{gray}\oval(11,  3.1)}

\put(-2.8,2.6){\color{red}{strongly miserable games = pet games}}
\put(-1,  1.9){\color{green}{miserable games}}
\put(-2,  1.3){\color{blue}{tame games = $t$-miserable games}}
\put(-3,  .7){domestic games = weakly miserable games}

\put(.5,3.8){\color{gray}\oval(7, 1.5)}
\put(-1,4){\color{gray}{forced games}}

\put(.5,3.8){\color{cyan}\oval(8.5, 3.1)}
\put(-1.3,4.8){\color{cyan}{returnable games}}

\end{picture}
\end{center}
\caption{The diagram of containments.} 
\label{Inc}
\end{figure}
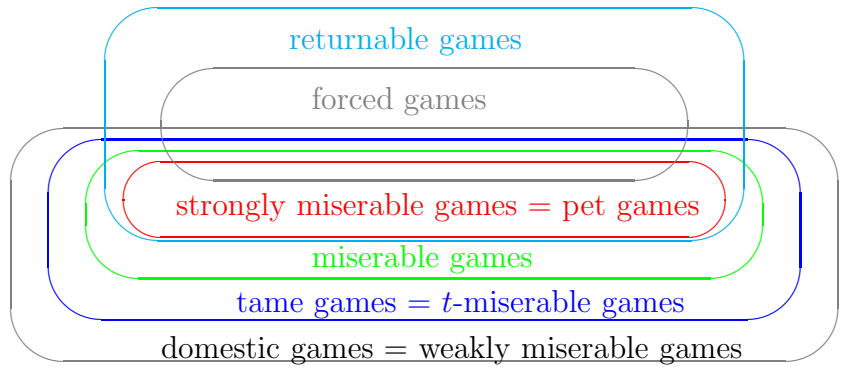

The following concept will be instrumental.
Given a position $x$ of a game $G$, we denote by $d(x)$ the greatest number of successive
moves from  $x$  to the terminal position.
Let us denote by  $G_{[x]}$ the subgame of $G$  defined by the initial position  $x$.
Obviously, $G_{[x]}$ contains  $x$  and all positions that can be reached from  $x$
(by one or several moves; recall that this set is finite) and all arcs between these positions.

\subsection{Domestic games and weakly miserable games coincide}

\begin{lemma}\label{V01V10}
In a domestic game, from each $(1,0)$-position there is a move
to a $(0,1)$-position and from each non-terminal $(0,1)$-position there is a move to a $(1,0)$-position.
\end{lemma}

\begin{proof}
From each $(1,0)$-position (resp.,~non-terminal $(0,1)$-position) there is a move to a $(0,k)$-position
(resp.,~$(k,0)$-position); obviously,  $k \neq 0$  (resp.,~$k \neq 1$).
Furthermore, $k \leq 1$ since the game is domestic.
\end{proof}



\begin{theorem}\label{D=WM}
A game is weakly miserable if and only if it is domestic.
\end{theorem}

\begin{proof}
\item [$(\Rightarrow)$]
Assume that $G$ is weakly miserable but not domestic.
Let  $x$  be a  $(0,k)$-position with  $k \geq 2$  for which   $d(x)$  takes the smallest possible value.
Then, there is a move from  $x$  to a  $(k',0)$-position  $x'$.
Since $d(x') < d(x)$, from our assumption we conclude that  $G_{[x']}$  is domestic and, hence, $k' \leq 1$.
Furthermore,  $k' = \G(x') \neq 0$ since  $\G(x) = 0$ and  $x$  is movable to  $x'$; hence,  $k' = 1$.
Thus, $x'$ is a $(1,0)$-position  and  $(\mathfrak{b})$  fails for  $x$.
Note that  $(\mathfrak{a})$ does not hold for $x$  either.

Similarly,  $x$  is movable to no position $y$  with $\G(y) = 0$, because  $\G(x) = 0$.
Therefore,  $\mathfrak(c)$, $\mathfrak(c_0)$, and $\mathfrak(c_1)$  fail for  $x$, resulting in a contradiction.
Thus, $G$ is domestic.

The case when  $x$ is a $(k,0)$-position, rather than $(0,k)$-position, is similar.

\item [$(\Leftarrow)$]
Assume that $G$ is domestic.
If  $x$ is a swap position, then $\mathfrak{(a)}$ holds for $x$.
If $x$ is a $(0,0)$-position or a $(1,1)$-position, then $\mathfrak{(b)}$  holds for $x$.
If $x$ is an $(a,b)$-position such that $\operatorname{max}(a,b) \geq 2$, then $\operatorname{min}(a,b) \geq 1$, because $G$ is domestic.

Without loss of generality, assume that $a \leq b$.
Since $a \geq 1$ and $b \geq 1$, there is a move from  $x$ 
to a $(0,i)$-position $y$ and to a $(j,0)$-position $z$.
Then,  $i \leq 1$ and $j \leq 1$, because  $G$ is domestic.
If $i = 1$ and $j = 1$, then $(\mathfrak{c})$ holds for $x$.
Otherwise,  $x$  is  movable to a $(0,0)$-position$^{(\star)}$.

If $(\mathfrak{b})$  fails for $x$, then $x$ is movable to either a $(0,1)$-position or a $(1,0)$-position$^{(\star\star)}$.
By $^{(\star)}$ and $^{(\star\star)}$, either $(\mathfrak{c_0})$ or $(\mathfrak{c_1})$ holds for  $x$.
Hence, the game is weakly miserable.
\end{proof}


\subsection{Tame games and $t$-miserable games coincide}


\begin{theorem} \label{tame-iff}
A game $G$ is tame if and only if it is $t$-miserable.
\end{theorem}

\begin{proof}
\item [$(\Rightarrow)$]
Let us assume that $G$ is tame and prove that
for every position $x$ at least one of three properties
$(\mathfrak{a_0}$), $(\mathfrak{c})$, $(\mathfrak{e})$ holds.

Furthermore,  $(\mathfrak{a_0})$ holds for $x$  if
$x$  is either a swap, or a $(0,0)$-positionor a $(1,1)$-position.
Assume that $x$ is a $(k,k)$-position for some $k \geq 2$.
By Lemma  \ref{SG} and its mis\`ere version, there are moves from $x$:

\smallskip
\noindent
to a $(0,i')$-position $x'$, to a $(1,i'')$-position $x''$, to a $(i''',0)$-position $x'''$, and to a $(i'''',1)$-position $x''''$.

\smallskip
\noindent
Furthermore, $\operatorname{max}(i',i'',i''',i'''') \leq 1$, since the game is tame.

\smallskip

If $i' = 1$ and $i'' = 0$, $(\mathfrak{c})$ holds.
If $i' = 0$ and $i'' = 1$, $(\mathfrak{e})$ holds.
If $i' = 1$ and $i'' = 1$, we consider $x'''$.
If $i''' = 1$, $(\mathfrak{c})$ holds; otherwise, $(\mathfrak{e})$ holds.
If $i' = 0$ and $i'' = 0$, consider $x''''$.
If $i'''' = 1$, $(\mathfrak{e})$ holds; otherwise, $(\mathfrak{c})$ holds.

\item [$(\Leftarrow)$]
Let us assume that ($\mathfrak{a_0}$), or $(\mathfrak{c})$, or $(\mathfrak{e})$ holds for every position
and prove by induction on $d(x)$ that each $x$ is either a $(k,k)$-position for some  $k \geq 0$  or a swap position. Note that the claim holds when  $d(x) \leq 1$.
Indeed, $d(x) = 0$  if and only if position  $x$ is terminal;
in this case  $x$  is a $(0,1)$-position.
Furthermore, $d(x) = 1$  if and only if every move from  $x$
results in a terminal position; in this case  $x$ is a $(1,0)$-position.


Let us proceed by induction.
Assume that the claim holds for every position $x$  with $d(x) \leq n$, for some $n \geq 1$, and
prove it for  $x$  with  $d(x) = n+1$.

Assume that $(\mathfrak{a_0})$ fails for an $(a,b)$-position  $x$.
Then, obviously, $a \geq 2$  or  $b \geq 2$.
Without loss of generality, assume that $a \geq 2$ and consider two sets
\[M = \{\G(y) \mid \text{ $y$ is a option of $x$}\} \text{ and } M^- = \{\G^-(y) \mid \text{ $y$ is a option of $x$}\}.\]

If $(\mathfrak{c})$  or  $(\mathfrak{e})$ holds for $x$, both $M$ and $M^-$ contain both $0$ and $1$.
Furthermore, if $y$ is a option of $x$ and $y \not \in V_{0,1} \cup  V_{1,0} \cup V_{0,0} \cup V_{1,1}$, then
$y$  is a  $(k,k)$-position for some $k \geq 2$ by the inductive hypothesis.
Therefore, $M = M^-$, implying that
$\G(x) = \operatorname{mex}(M) = \operatorname{mex}(M^-) = \G^-(x)$ and, hence,  $x$ is a $(k,k)$-position for some  $k \geq 0$.
\end{proof}


\subsection{Miserable games are tame}

\begin{theorem} \label{MisT}
A miserable game is tame.
\end{theorem}

This statement was announced in \cite{Gur07} and shown in \cite{Gur11, Gur120}.
Here we provide simpler arguments.

\begin{proof}
Assume that $G$ is miserable and prove by induction on $d(x)$ that
every position $x$ is either a swap position or a $(k,k)$-position for some $k \geq 0$.


The case  $d(x) \leq 1$ was already considered in the proof of Theorem \ref{tame-iff} above (if $d(x) = 0$  then $x$  is a $(0,1)$-position; if $d(x) = 1$ then $x$ is a $(1,0)$-position).

Let us assume that the claim holds for every position $x$  with  $d(x) \leq n$, for some $n \geq 1$, and
prove that it holds for every position $x$  with $d(x) = n+1$.

Since $G$ is miserable,  $(\mathfrak{a})$, or $(\mathfrak{b})$, or $(\mathfrak{c})$  holds for $x$.
\begin{enumerate} [(i)] \itemsep0em
\item If $(\mathfrak{a})$ holds, $x$ is a swap position and we are done.
\item If $(\mathfrak{b})$  holds, by the inductive hypothesis, each option  $y$  of  $x$  is a  $(k_y,k_y)$-position for some $k_y \geq 0$. Therefore, $x$ is a $(k,k)$-position in which
    $$k = \operatorname{mex}\{k_y \mid \text{ $y$ is a option of $x$}\}.$$
\item If $(\mathfrak{c})$ holds, by the inductive hypothesis, each option  $y$  of  $x$ is either a swap position or a  $(k_y,k_y)$-position for some $k_y \geq 0$. Therefore, $x$ is a $(k,k)$-position in which
    $$k = \operatorname{mex}\{0,1,k_y \mid \text{ $y$ is a option of $x$ and $y$ is a $(k_y,k_y)$-position}\}.$$
    Note that in this case, $k \geq 2$.
\end{enumerate}
\end{proof}

\smallskip

Figure \ref{TnM} provides a tame game that is not miserable
showing that the containment of Theorem \ref{MisT} is strict.

\begin{figure}[ht]
\begin{center}
\begin{tikzpicture}[style={draw=blue,thick,circle,inner sep=0pt}] 
 \node at (0in,0in) [shape=circle,minimum size=.6cm,draw=black,thick] (A) {0,1};
 \node at (1in,0in) [shape=circle,minimum size=10pt,draw=black,thick] (B) {1,0};
 \node at (2in,0in) [shape=circle,minimum size=10pt,draw=black,thick] (C) {2,2};
 \node at (3in,0in) [shape=circle,minimum size=10pt,draw=black,thick] (D) {0,0};
 \node at (4in,0in) [shape=circle,minimum size=10pt,draw=black,thick] (E) {1,1};
 \node at (5in,0in) [shape=circle,minimum size=10pt,draw=black,thick] (F) {3,3};
\draw [->] (F) to (E);
\draw [->] (E) to (D);
\draw [->] (D) to (C);
\draw [->] (C) to (B);
\draw [->] (B) to (A);
\draw [->] (C) to [out=160,in=20] (A);
\draw [->] (F) to [out=160,in=20] (D);
\draw [->] (F) to [out=155,in=25] (C);
\draw [->] (F) to [out=150,in=30] (A);
\end{tikzpicture}
\end{center}
\caption{\label{TnM}
This game is tame but not miserable, since
$(\mathfrak{a})$, $(\mathfrak{b})$, and $(\mathfrak{c})$
fail for the initial position.}
\end{figure}
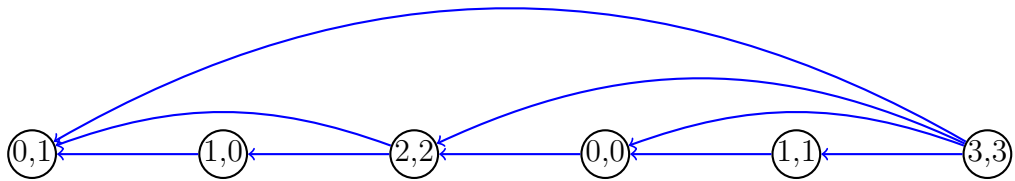

\subsection{Pet games and strongly miserable games coincide}
There pet games can be characterized in many equivalent ways;
the following list was suggested in \cite{Gur120}.

\begin{theorem} \label{SM=P}
The following properties of a game $G$ are equivalent.
\begin{enumerate} \itemsep0em
\item [\rm{(i)}]   $G$ is strongly miserable.
\item [\rm{(ii)}]  $G$ is pet.
\item [\rm{(iii)}] $G$ has no $(0,0)$-position.
\item [\rm{(iv)}]  $G$ has neither $(0,0)$-position nor  $(1,1)$-position.
\item [\rm{(v)}]   If $\G(x) = 0$ and $x$ is not terminal then $x$ is movable to some $x'$ with $\G(x') = 1$.
\item [\rm{(vi)}]  If $\G^-(x) = 0$ then  $x$ is movable to some $x'$ with $\G^-(x') = 1$.
\end{enumerate}
\end{theorem}


Interestingly, property \rm{(v)}, claiming that
any non-terminal $0$-position is movable to a $1$-position, was
introduced (for some other purposes) already in 1974 by Ferguson \cite{Fer74}
who proved that it holds for all subtraction games; see Section~\ref{S.App}.

Some proofs were given in \cite{Gur120}. Here we give the complete analysis.

\begin{proof}[Proof of Theorem \ref{SM=P}.]

\item [$\rm{(i)} \Rightarrow \rm{(ii)}$.]
Every strongly miserable game is miserable and hence tame, by Theorem \ref{MisT}.
It remains to show that $G$  has neither $(0,0)$-position no $(1,1)$-position.
Indeed, assume that  $x$  is such a position.
Then, properties $(\mathfrak{a})$ and $(\mathfrak{c})$  of Definition \ref{D.WTMS} fail for $x$, which is contradiction.

\item [$\rm{(ii)} \Rightarrow \rm{(i)}$.]
Let  $x$  be a non-swap position of  $G$.
Since $G$ is pet,  $x$  is a  $(k,k)$-position for some $k \geq 2$.
By Lemma  \ref{SG} and its mis\`ere version, there are moves from  $x$ to a $(0,i)$-position and to a $(j,0)$-position.
Since  $G$  is pet, $i = j = 1$. Thus, $(\mathfrak{c})$ holds for $x$.

\item [$\rm{(ii)} \Rightarrow \rm{(iii)}$.]
This implication is straightforward.

\item [$\rm{(iii)} \Rightarrow \rm{(ii)}$.]
Assume that  $G$  has no $(0,0)$-position and prove by induction on  $d(x)$  that every position $x$
is either a swap position or a $(k, k)$-position for some  $k \geq 2$.
Standardly, the claim can be verified for the case  $d(x) \leq 1$.

Suppose that position  $x$  is a counterexample with the smallest value of  $d(x)$.
The following case analysis results in a contradiction:
\begin{enumerate} [(a)] \itemsep0em
\item Case 1: $x$ is a  $(1,1)$-position.
    Then $x$  is movable to a $(0,e)$-position $x_1$ with $e \neq 1$.
    Since $d(x_1) < d(x)$, our choice of $x$ implies $e = 1$, which is impossible.
\item Case 2: $x$ is a $(0, a)$-position (the case where $x$ is a $(a, 0)$-position is treated similarly) with $a \geq 2$.
    Then $x$ is movable to some $(e,1)$-position $x_2$  with $e \neq 0$.
    Since $d(x_2) < d(x)$, our choice of $x$ implies $e = 0$, which is impossible.
\item Case 3: $x$ is a $(b,c)$-position with $1 \leq b < c$.
    Then, there must be three options $x_3, x_4, x_5$  of  $x$  such that
    \begin{itemize} \itemsep0em
    \item $x_3$ is a $(0,i)$-position for some $i \geq 1$,
    \item $x_4$ is a $(j,0)$-position for some $j \geq 1$, and
    \item $x_5$ is a $(k,b)$-position for some $k \geq 0$
    \end{itemize}
    By the choice of $x$, we have $j=1$, and hence, $b \geq 2$.
    Furthermore, since  $b \geq 2$ and $d(x_5) < d(x)$, 
    we have  $k = b$ or equivalently $\G(x_5) = \G(x)$, which is impossible.
\end{enumerate}

\item [$\rm{(iii)} \Leftrightarrow \rm{(iv)}$.] We already proved that $\rm{(iii)} \Rightarrow \rm{(ii)}$.
Furthermore, $\rm{(ii)} \Rightarrow \rm{(iv)}$ results immediately from the definition of pet games.
Thus, $\rm{(iii)} \Rightarrow \rm{(iv)}$  holds.

\item [$\rm{(ii)} \Rightarrow \rm{(v)} (\text{ resp.,} \rm{(vi)})$.]
Assume that  $G$ is pet. Let  $x$  be a position with  $\G(x) = 0$ (resp.,~$\G^-(x) = 0$).
Since $G$ is pet, $x$  must be a $(0,1)$-position (resp.,~$(1,0)$-position).
Since  $x$  is not a terminal position, it is movable to a $(l, 0)$-position
(resp.,~to a $(0,l)$-position) for some $l$. Since  $G$ is pet, we have $l = 1$, as required.

\item [$\rm{(v)} \Rightarrow \rm{(ii)}$.]
Assume that  \rm{(v)}  holds for a game  $G$  that is not pet.
Then, $G$  contains a position $x$  that is neither swap nor a $(k, k)$-position for any $k \geq 2$.
Due to symmetry, we can assume that $x$ is either
\begin{enumerate}\itemsep0em
\item a $(0, 0)$-position, or
\item a $(1, 1)$-position, or
\item an  $(m, n)$-position with $0 \leq m < n$ and $n \geq 2$.
\end{enumerate}

As  usual, let us choose such an $x$ with the smallest $d(x)$. Then,
\begin{enumerate}\itemsep0em
\item [$(\star)$]  every position $x'$ with $d(x') < d(x)$  is
a swap or a $(k,k)$-position for some $k \geq 2$
\end{enumerate}
In case (1) (resp.,~(2)), $x$  is movable to a position  $x'$  with  $\G(x') = 1$, by \rm{(v)}
(resp.,~$\G(x')=0$, by the SG Theorem).
Then, $x'$ is a $(1,0)$-position(resp.,~a $(0,1)$)-position, by  $(\star)$ and the assumption $d(x') < d(x)$.
Hence,  $\G^-(x') = 0 = \G^-(x)$ (resp.,~$\G^-(x') = 1 = \G^-(x)$), resulting in a contradiction.

Since $n \geq 2$, in case (3) there are moves  $x \rightarrow x'$ and $x \rightarrow x''$ such that
$\G^-(x') = 0$ and $\G^-(x'') = 1$, by Lemma  \ref{SG} and its mis\`ere version.
Since $d(x') < d(x)$ and $d(x'') < d(x)$, by $(\star)$ we conclude that
$x'$ and $x''$ are a $(1,0)$-positionand $(0,1)$-position, respectively.
Hence, $m \geq 2$. Since $\G^-(x) = n > m$, there exists a move $x \rightarrow x'''$ such that
$\G^-(x''') = m$, that is, $x'''$ is a $(r,m)$-position for some $r$.
Since $d(x''') < d(x)$ and $m \geq 2$, by  $(\star)$  we have  $r = m$.
Thus, that $\G(x) = m = \G^-(x''')$, resulting in a contradiction.

\item [$\rm{(vi)} \Rightarrow \rm{(ii)}$.] This case is similar to the case $(v) \Rightarrow (ii)$.
\end{proof}


\begin{proposition} \label{SM.R}
Strongly miserable games are returnable.
\end{proposition}

\begin{proof}
Suppose that $G$ is a strongly miserable game and  $x$ is its  $(0,1)$-position
(resp.,~$(1,0)$-position).
If  $d(x) \leq 1$, we are done. Assume that $d(x) \geq 2$.
Then each option $x'$ of $x$ is an $(i,j)$-position with  $i > 0$
(resp.,~$j > 0$), by Lemma  \ref{SG} (resp.,~by its mis\`ere version).
Hence, $x'$  is movable to a $(0,k)$-position
(resp.,~$(k,0)$-position). 
Then,  $k = 1$, since $G$ is strongly miserable (pet).
Thus,  $G$ is returnable.
\end{proof}


\section{Constructive characterizations of domestic, tame, miserable, and strongly miserable games}
\label{S.Equi}

\subsection{A general plan}

We could make use of Definitions \ref{D.DTP}  and \ref{D.WTMS} to verify whether
a game is miserable or strongly miserable, but to do so we have to know its swap positions.
It may be even more difficult to verify membership in the other considered classes, because
the sets  $V_{0,0}$ and/or  $V_{1,1}$  become also involved.
Since the SG values are defined recursively, it looks difficult
to guarantee in advance that a given subset
contains all, for example, $(0, 1)$-positions; see Definition \ref{D.WTMS}.

To avoid this problem and obtain constructive characterizations,
we will modify Definitions \ref{D.DTP}, \ref{D.WTMS}  and obtain
Theorems \ref{SMis.ss4}, \ref{Mis.ss4}, \ref{tame.ss4}, \ref{Dos.ss4}
characterizing strongly miserable (pet), miserable, $t$-miserable
(tame), weakly miserable (domestic) games, respectively.
In these theorems,  sets  $V_{0,1}$, $V_{1,0}$, $V_{0,0}$, $V_{1,1}$
of Definition \ref{D.WTMS} are replaced by some ``abstract" sets
$V'_{0,1}$, $V'_{1,0}$, $V'_{0,0}$, $V'_{1,1}$.
Requiring (almost) the same properties from these sets, we
characterize all above classes and  show that
the old and new sets are equal, that is,
$V'_{i,j} = V_{i,j}$  for all  $i,j \in \{0,1\}$.

We will prove only Theorem \ref{SMis.ss4};
the remaining three theorems can be proven in a similar way
and we leave them to the reader.

\subsection{Strongly miserable games}
Let us begin with the strongly miserable (pet) games.

\begin{theorem} \label{SMis.ss4}
A game $G$ is strongly miserable if and only if there exist
two disjoint sets  $V'_{0,1}$
and $V'_{1,0}$  satisfying the following conditions:
\begin{itemize}\itemsep0em
\item [\rm{(i)}]   both sets are independent, that is, there is no move between two positions of one set;
\item [\rm{(ii)}]  $V'_{0,1}$ contains all terminal positions, $V_T \subseteq V'_{0,1}$;
\item [\rm{(iii)}] $V'_{0,1} \setminus V_T$  is movable to  $V'_{1,0}$;
\item [\rm{(iv)}]  $V'_{1,0}$  is movable to  $V'_{0,1}$;
\item [\rm{SM(v)}] exactly one of the next two conditions holds for each position $x$:
    \begin{itemize} \itemsep0em
    \item [$(\mathfrak{a}')$] $x \in V'_{0,1} \cup V'_{1,0}$;
    \item [$(\mathfrak{c}')$] $x$ is movable to $V'_{0,1}$ and to $V'_{1,0}$.
    \end{itemize}
\end{itemize}
Moreover, if all above conditions hold then  $V'_{0,1} = V_{0,1}$  and  $V'_{1,0} = V_{1,0}$.
\end{theorem}

\begin{proof}
The ``only if" part is straightforward, by setting $V'_{0,1} = V_{0,1}$  and  $V'_{1,0} = V_{1,0}$. Let us prove the ``if" part.
Actually, it is enough to prove that  $V'_{0,1} = V_{0,1}$ and $V'_{1,0}  = V_{1,0}$. It then follows from condition \rm{SM(v)} that the game does not have $(0,0)$-position, and so it is strongly miserable by Theorem  \ref{SM=P}.

As  usual, we proceed by induction on  $d(x)$  to show the following claims:
\begin{itemize}\itemsep0em
\item [$(1)$] If $x$ is a $(0,1)$-position, then $x \in V'_{0,1}$;
\item [$(2)$] If $x$ is a $(1,0)$-position, then $x \in V'_{1,0}$;
\item [$(3)$] If $x$ is not a swap position then $x$ is a $(k,k)$-position for some $k \geq 2$ and, moreover, $(\mathfrak{c'})$ holds for $x$.
\end{itemize}

If $d(x) = 0$ then  $x$  is a terminal position and  $(1)$ holds, since $V'_{0,1}$ contains $V_T$.
If $d(x) = 1$ then $x$ is a $(1,0)$-position that is movable to terminal position.
Moreover, there are no other moves from $x$.
In particular, it means that there is no move from $x$ that terminates in $V'_{1,0}$.
The condition SM\rm{(v)} implies that $x \in  V'_{0,1} \cup V'_{1,0}$.
By 
\rm{(i)}, $x \notin  V'_{0,1}$ and so $x \in V'_{1,0}$. Thus $(2)$ holds for $x$.

The claims $(1) - (3)$ are standardly verified for $d(x) =0 $  and $d(x) = 1$.
Let us assume that it holds for every position $x$ with $d(x) \leq n$ for some $n \geq 1$ and prove it for $x$ such that $d(x) = n+1$.

\begin{itemize}\itemsep0em
\item [$(1)$]
 Let $x$  be a  $(0,1)$-position.
 Then, $x$ is not movable to  $V'_{0,1} \cap G_{[x]}$, because each position of this set is
 a $(0,1)$-position, by the inductive hypothesis on $(1)$, meaning $x$ is not movable to $V'_{0,1}$.
 From this fact and SM\rm{(v)} it follows that $x \in V'_{0,1} \cup V'_{1,0}$. We show that $x \notin V'_{1,0}$.

Assume for contradiction that $x \in V'_{1,0}$. It follows from \rm{(iv)} that  $x$ is movable to a position  $y \in V'_{0,1}$.
By induction, if $y \in (G_{[x]} \cap V'_{0,1}) \setminus \{x\}$ then $y$ is a $(0,1)$-position.
But $x$ is a $(0,1)$-position too and, hence, it cannot be movable to such $y$. This give a contradiction.

Thus, $x \notin V'_{1,0}$, implying that $x \in V'_{0,1}$ or, equivalently, that $(1)$ holds.

\item [$(2)$] Similarly, assuming that  $x$ is a $(1,0)$-position.  We can show that  $x \in V'_{1,0}$.

\item [$(3)$] Assume that  $x$  is not a swap position. We show that  $(3)$ holds.
First, note that $x$ is neither a $(0,0)$-position nor a $(1,1)$-position as well,
because  $(\mathfrak{a}')$ or $(\mathfrak{c}')$ holds for  $x$.

Let $x$ be a $(k,l)$-position such that either $k \geq 2$ or $l \geq 2$.
Without loss of generality, assume that $k \geq 2$.
Then, $(\mathfrak{a}')$ fails for $x$ and, hence, $(\mathfrak{c}')$ holds.
It follows that $x$ is movable to a position $x'$ in $V'_{0,1}$ and to a position $x''$ in $V'_{1,0}$.
It remains to show that $l = k$.

Let us consider two sets
\[M = \{\G(y)\mid \text{ $y$ is a option of $x$} \} \text{ and }  M^- = \{\G^-(y)\mid \text{ $y$ is a option of $x$} \}.\]

We have $\{0,1\} \subseteq M$ and $\{0,1\} \subseteq M^-$, since both $x'$ and $x''$ are options of $x$.
Moreover, by the inductive hypothesis, if an option $y$ of $x$ is not a swap position then  $y$ is a $(m,m)$-position. Therefore, $M = M^-$ and, hence,
\[k = \G(x) = \operatorname{mex}(M) = \operatorname{mex}(M^-) =\G^-(x) = l.\]
\end{itemize}
\end{proof}


\subsection{Miserable games}


Miserable games can be characterized in a similar  way;
only property \rm{SM(v)} of Theorem \ref{SMis.ss4} is slightly changed.

\begin{theorem} \label{Mis.ss4}
A game  $G$ is miserable if and only if
there exist two disjoint sets  $V'_{0,1}$ and $V'_{1,0}$ satisfying 
$(\rm{i}) - (\rm{iv})$ of Theorem \ref{SMis.ss4}  and
every position $x$ satisfies at least one of the following three conditions:

    \begin{itemize} \itemsep0em
    \item [$(\mathfrak{a}')$] $x \in V'_{0,1} \cup V'_{1,0}$;
    \item [$(\mathfrak{b}')$] $x$ is not movable to $V'_{0,1} \cup V'_{1,0}$;
    \item [$(\mathfrak{c}')$] $x$ is movable to $V'_{0,1}$ and to $V'_{1,0}$.
    \end{itemize}
Moreover, if all above conditions hold then  $V'_{0,1} = V_{0,1}$ and $V'_{1,0} = V_{1,0}$.
\end{theorem}



\subsection{Tame games}


Similarly, we characterize tame games as follows.

\begin{theorem} \label{tame.ss4}
A game is tame if and only if there exist four disjoint sets
$V'_{0,1}$, $V'_{1,0}$, $V'_{0,0}$, $V'_{1,1}$ satisfying the following conditions:
\begin{itemize} \itemsep0em
\item [\rm{(i)}]   all four sets are independent;
\item [\rm{(ii)}]  $V'_{0,1}$  contains the terminal position, , $V_T \subseteq V'_{0,1}$;
\item [\rm{(iii)}] if $x \in V'_{0,1}$ then $x$ is movable to $V'_{1,0}$ but not to $V'_{0,0} \cup V'_{1,1}$;
\item [\rm{(iv)}]  if $x \in V'_{1,0}$ then $x$ is movable to $V'_{0,1}$ but not to $V'_{0,0} \cup V'_{1,1}$;
\item [\rm{(v)}]   if $x \in V'_{0,0}$ then  $x$ is not movable to $V'_{0,1} \cup V'_{1,0}$;
\item [\rm{(vi)}]  if $x \in V'_{1,1}$ then $x$ is movable to $V'_{0,0}$ but not to $V'_{0,1} \cup V'_{1,0}$;
\item [\rm{(vii)}] if $x \not\in V'_{0,1} \cup V'_{1,0} \cup V'_{0,0}$ then
$x$ is movable to $V'_{0,1} \cup V'_{1,0} \cup V'_{0,0}$.
\item [\rm{T(viii)}] Every position $x$ satisfies at least one of the following three conditions:
    \begin{itemize} \itemsep0em
    \item [$(\mathfrak{a_0}')$] $x \in V'_{0,1} \cup V'_{1,0} \cup V'_{0,0} \cup V'_{1,1}$;
    \item [$(\mathfrak{c}')$]   $x$ is movable to $V'_{0,1}$ and to $V'_{1,0}$;
    \item [$(\mathfrak{e}')$]   $x$ is movable to $V'_{0,0}$ and to $V'_{1,1}$.
    \end{itemize}
\end{itemize}
Moreover, $V'_{0,1} = V_{0,1}$, $V'_{1,0} = V_{0,1}$, $V'_{0,0} = V_{0,0}$, and $V'_{1,1} = V_{1,1}$
whenever all above conditions hold.
\qed
\end{theorem}

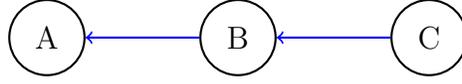
\begin{figure}[ht]
\begin{center}
\begin{tikzpicture}[style={draw=blue,thick,circle,inner sep=0pt}]   
 \node at (0in,0in) [shape=circle,minimum size=1cm,draw=black,thick] (X) {A};
 \node at (1in,0in) [shape=circle,minimum size=1cm,draw=black,thick] (Y) {B};
 \node at (2in,0in) [shape=circle,minimum size=1cm,draw=black,thick] (Z) {C};
\draw [->] (Z) to (Y);
\draw [->] (Y) to (X);
\end{tikzpicture}
\end{center}
\caption{\label{remark}$V_{0,1} \neq V'_{0,1}$, although conditions \rm{(i)} - \rm{(viii)} of Theorem \ref{tame.ss4} hold.}
\end{figure}

One may ask, whether conditions \rm{(i)} - \rm{(vii)} of Theorem \ref{tame.ss4} themselves
result in equalities
$V'_{0,1} = V_{0,1}$ and $V'_{1,0} = V_{1,0}$ too.
This is not the case.
The game in Figure \ref{remark} provides a counterexample with setting
$V'_{0,1} = A$, $V'_{1,0} = B$, and $C \in V_{0,1} \setminus V'_{0,1} \neq \emptyset$.




\subsection{Domestic games}

Finally, a similar characterization holds for the domestic games.


\begin{theorem} \label{Dos.ss4}
A game is domestic if and only if there exist three disjoint sets
$V'_{0,1}$, $V'_{1,0}$, and $V'_{0,0}$ such that the following conditions hold:
\begin{itemize}\itemsep0em
\item [\rm{(i)}]  all three sets are independent;
\item [\rm{(ii)}] $V'_{0,1}$ contains all terminal positions;
\item [\rm{(iii)}] if $x \in V'_{0,1}$ is non-terminal, $x$ is movable to $V'_{1,0}$ but not to $V'_{0,0}$;
\item [\rm{(iv)}]  If $x \in V'_{1,0}$, $x$ is movable to $V'_{0,1}$ but not to $V'_{0,0}$;
\item [\rm{(v)}]   If $x \in V'_{0,0}$, $x$ is not movable to $V'_{0,1} \cup V'_{1,0}$;
\item [\rm{(vi)}]  If $x \notin V'_{0,1} \cup V'_{1,0} \cup V'_{0,0}$, $x$ is movable to $V'_{0,1} \cup V'_{1,0} \cup V'_{0,0}$;
\item [\rm{D(vii)}]  every position $x$ satisfies at least one of conditions
    \begin{itemize} \itemsep0em
    \item [$(\mathfrak{a}')$] $x \in V'_{0,1} \cup V'_{1,0}$;
    \item [$(\mathfrak{b}')$] $x$ is not movable to $V'_{0,1} \cup V'_{1,0}$;
    \item [$(\mathfrak{c}')$] $x$ is movable to $V'_{0,1}$ and to $V'_{1,0}$;
    \item [$(\mathfrak{c_0}')$] $x$ is movable to $V'_{0,1}$ and to $V'_{0,0}$;
    \item [$(\mathfrak{c_1}')$] $x$ is movable to $V'_{1,0}$ and to $V'_{0,0}$.
    \end{itemize}
\end{itemize}
Moreover, if all above conditions hold then
$V'_{0,1} = V_{0,1}$, $V'_{1,0} = V_{0,1}$, and $V_{0,0} = V_{0,0}$.
\qed
\end{theorem}



\section{Sums of games}
\label{S.Sum}

We say that a class of games is preserved under summation
if the sum of games from this class belongs to it too.
In this section, we show the classes of
tame, miserable, miserable and forced, miserable and returnable games are preserved under summation.
For the tame games, this property was claimed by Conway in \cite{Con76} and proven in \cite{CGT};
we suggest a simpler proof.

In contrast, the classes of domestic (weakly miserable) and of pet
(strongly miserable) games are not preserved under summation.
Already the classic  $n$-pile  {\sc Nim} is a counterexample for the second case.
Indeed, one-pile {\sc Nim} is pet but the  $n$-pile {\sc Nim},
which is the sum of  $n$  one-pile {\sc Nim} games, is not whenever  $n > 1$;
see Subsection \ref{nim} for more details.

The sum of domestic games may be not domestic; Figure \ref{Sodo} gives an example.
\begin{figure}[ht]
\begin{center}
\begin{tikzpicture} [style={draw=blue,thick,circle,inner sep=0pt}]   
\node at (-2.1in,0in)(G) {G$_1$};
\node at (-2in,.5in) [shape=circle,minimum size=.6cm,draw=black,thick] (A) {0,1};
\node at (-2.25in,.5in) {A};
\node at (-2.3in,1in) [shape=circle,minimum size=.6cm,draw=black,thick] (B) {1,0};
\node at (-2.55in,1in) {B};
\node at (-2in,1.5in) [shape=circle,minimum size=.6cm,draw=black,thick] (C) {2,2};
\node at (-2.25in,1.5in) {C};
\node at (-2.3in,2in) [shape=circle,minimum size=.6cm,draw=black,thick] (D) {0,0};
\node at (-2.55in,2in) {D};
\node at (-2in,2.5in) [shape=circle,minimum size=.6cm,draw=black,thick] (E) {1,2};
\node at (-2.25in,2.5in) {E};
\draw [->] (E) to (D);
\draw [->] (E) to [out=-70,in=70] (A);
\draw [->] (D) to (C);
\draw [->] (C) to (B);
\draw [->] (C) to (A);
\draw [->] (B) to (A);
\node at (-1.5in,0in) (a) {};
\node at (-1.5in,3in) (b) {};
\draw [.] (a) to (b);
\node at (-1in,0in) (G2) {G$_2$};
\node at (-1in,.5in) [shape=circle,minimum size=.6cm,draw=black,thick] (X) {0,1};
\node at (-1.25in,.5in) {X};
\node at (-1in,1in) [shape=circle,minimum size=.6cm,draw=black,thick] (Y) {1,0};
\node at (-1.25in,1in) {Y};
\draw [->] (Y) to (X);
\node at (-.5in,0in)   (c) {};
\node at (-.5in,3in) (d) {};
\draw [.] (c) to (d);
\node at (  1in,  0in) (G12){$G_1+G_2$};
\node at (1in,0.5in) [shape=circle,minimum size=.6cm,draw=black,thick] (AX) {0,1};
\node at (.7in,0.5in) {AX};
\node at (.5in,1.0in) [shape=circle,minimum size=.6cm,draw=black,thick] (BX) {1,0};
\node at (.2in,1.0in) {BX};
\node at (.5in,1.5in) [shape=circle,minimum size=.6cm,draw=black,thick] (CX) {2,2};
\node at (.2in,1.5in) {CX};
\node at (.5in,2.0in) [shape=circle,minimum size=.6cm,draw=black,thick] (DX) {0,0};
\node at (.2in,2.0in) {DX};
\node at (.5in,2.5in) [shape=circle,minimum size=.6cm,draw=black,thick] (EX) {1,2};
\node at (.2in,2.5in) {EX};
\node at (1.5in,1.0in) [shape=circle,minimum size=.6cm,draw=black,thick] (AY) {1,0};
\node at (1.2in,1.0in) {AY};
\node at (1.5in,1.5in) [shape=circle,minimum size=.6cm,draw=black,thick] (BY) {0,1};
\node at (1.2in,1.5in) {BY};
\node at (1.5in,2.0in) [shape=circle,minimum size=.6cm,draw=black,thick] (CY) {3,3};
\node at (1.2in,2.0in) {CY};
\node at (1.5in,2.5in) [shape=circle,minimum size=.6cm,draw=black,thick] (DY) {1,1};
\node at (1.2in,2.5in) {DY};
\node at (1in,3.0in) [shape=circle,minimum size=.6cm,draw=black,thick] (EY) {0,3};
\node at (.7in,3.0in) {EY};
\draw [->] (BX) to (AX);
\draw [->] (CX) to (BX);
\draw [->] (CX) to [out=-60,in=100] (AX);
\draw [->] (DX) to (CX);
\draw [->] (EX) to (DX);
\draw [->] (EX) to  [out=-50,in=90] (AX);
\draw [->] (BY) to (AY);
\draw [->] (CY) to (BY);
\draw [->] (CY) to [out=-45,in=45] (AY);
\draw [->] (DY) to (CY);
\draw [->] (EY) to (DY);
\draw [->] (EY) to  [out=-90,in=130] (AY);
\draw [->] (AY) to (AX);
\draw [->] (BY) to (BX);
\draw [->] (CY) to (CX);
\draw [->] (DY) to (DX);
\draw [->] (EY) to (EX);
\end{tikzpicture}
\end{center}
\caption{Games $G_1$ and $G_2$ are domestic but their sum $G_1 + G_2$ is not.
Notation  $P(i,j)$  means that  $P$ is an $(i,j)$-position in a summand, while
$PQ(i,j)$ means that the sum  $PQ$  of  $P$ and $Q$  is an $(i,j)$-position.} \label{Sodo}
\end{figure}
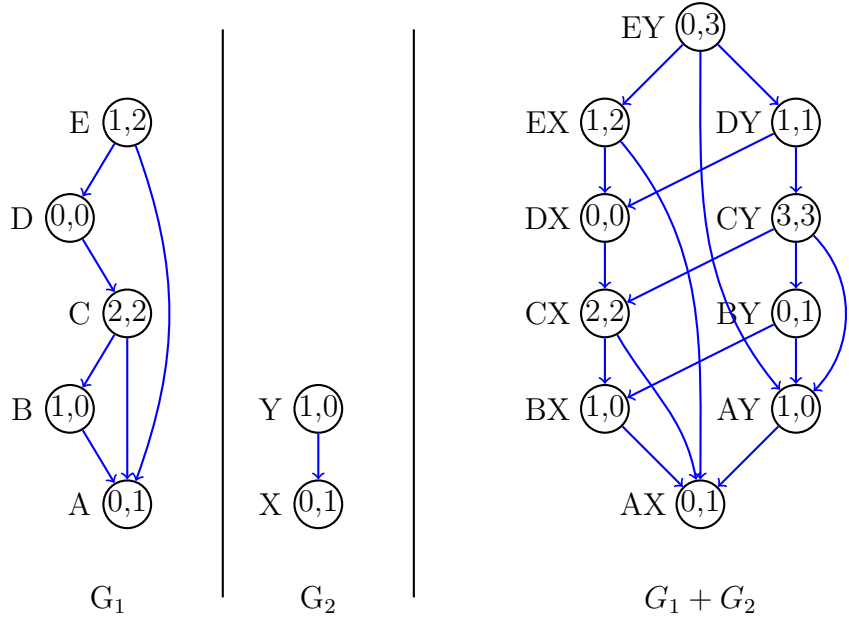

\subsection{The sum of tame games is tame}


Recall that a {\it swap} position is either a $(0,1)$-position or a $(1,0)$-position.
We will call two swap positions {\it opposite} if
one of them is a $(0,1)$-position while the other is a $(1,0)$-position, and
we will call them {\it parallel} otherwise.

\begin{theorem} 
\label{S.T}
If games  $G_1$ and $G_2$ are tame then their sum $G_1 + G_2$  is tame too.
Moreover, $x = (x_1, x_2)$ is a swap position of  $G_1 + G_2$  if and only if
$x_i$  is a swap position of $G_i$  for  $i = 1,2$.
Furthermore, $x$ is a $(1,0)$-position of $G$ if and only if either $x_1$ is a $(1,0)$-position in $G_1$ and $x_2$ is a $(0,1)$-position in $G_2$ or vise versa.
\end{theorem}

The first claim was stated (without a proof) in 1976 by Conway; see \cite{Con76} page 178.
A proof based on the genus theory appeared in \cite{CGT}.
Here we give an alternative proof based on the characterization
of tame games by Theorem \ref{tame.ss4}.


\smallskip

\begin{proof}  
For non-negative integers $i, j, k$ and $l$, denote
by $[(i,j),(k,l)]$ the set of positions  $x = (x_1,x_2)$ in the sum
$G_1 + G_2$ such that $x_1$ is an $(i,j)$-position in $G_1$ and $x_2$ is a $(k,l)$-position in $G_2$.
Let us set
\begin{align*}
V'_{0,1} = \{&[(0,1),(0,1)], [(1,0),(1,0)]\}, \\
V'_{1,0} = \{&[(0,1),(1,0)], [(1,0),(0,1)]\}, \\
V'_{0,0} = \{&[(0,1),(0,0)], [(0,0),(0,1)],[(n,n),(n,n)] \mid n \in \Z_{\geq 0}\}, \\
V'_{1,1} = \{&[(0,0),(1,0)], [(1,0),(0,0)],[(0,1),(1,1)],[(1,1),(0,1)], \\
             &[(n,n),(n+1,n+1)], [(n+1,n+1),(n,n)] \mid n = 2k,  k\in \Z_{\geq 0}\}
\end{align*}
Recall that $\Z_{\geq 0}$ denotes the set of non-negative integers.

It can be verified that the above four sets satisfy conditions \rm{(i)} - \rm{(vii)}  of Theorem \ref{tame.ss4}.
We now prove by induction on $d(x)$ that every position $x$ of the sum $G = G_1+G_2$ satisfies
(at least) one of the conditions
$(\mathfrak{a_0}'), (\mathfrak{c}'), (\mathfrak{e}')$ of Theorem \ref{tame.ss4} and so 
the sum is tame.

Note that in this proof, when we recall conditions
$(\mathfrak{a_0})$, $(\mathfrak{c})$, and $(\mathfrak{e})$ (resp.,~$(\mathfrak{a'_0})$, $(\mathfrak{c'})$, and $(\mathfrak{e'})$), we refer them in Definition \ref{D.WTMS} (resp.,~Theorem \ref{tame.ss4}).

By definition,  $x = (x_1,x_2)$ is a terminal position of the sum  $G = G_1 + G_2$  if and only if
each $x_i$ is a terminal position of the summand  $G_i$,  for  $i = 1,2$.
Hence,
$(\mathfrak{a_0}')$ holds for $(x_1,x_2)$.
If $d(x_1,x_2) = 1$ then either $d(x_1) = 0$ ($x_1$ is terminal) and
$d(x_2) = 1$ or vise versa and so $(x_1,x_2) \in V'_{1,0}$, meaning $(\mathfrak{a_0}')$ holds for $(x_1,x_2)$.

We assume that at least one of the conditions
$(\mathfrak{a_0}')$, $(\mathfrak{c}')$, $(\mathfrak{e}')$ holds
for every position $(x_1,x_2)$ in $G$ such that $d(x_1,x_2) \leq n$ for some $n \geq 1$
and will show that at least one of these conditions
holds for each position $(x_1,x_2)$ in $G$ such that $d(x_1,x_2) = n+1$.

Suppose that $(\mathfrak{a_0}')$ fails for $x = (x_1,x_2)$.
Then there exists a move from $x$ to a position $x' \in V'_{0,1} \cup V'_{1,0} \cup V'_{0,0}$,
by  \rm{(vii)} of Theorem \ref{tame.ss4}.
Assume such move $x_1 \rightarrow x'_1$ is made in $G_1$.

\begin{enumerate}
\item Case  $x' = (x'_1,x_2) \in V'_{0,1}$.
In this case  $x'_1$ and $x_2$ are two parallel swap positions.
Since $x_1$ is movable to the swap position $x'_1$, condition $(\mathfrak{a_0})$  fails for $x_1$
and, hence, $(\mathfrak{c})$ or $(\mathfrak{e})$ holds for $x$, since $G_1$ is tame.
    \begin{itemize}\itemsep0em
    \item [$(a)$]  If $(\mathfrak{c})$ holds for $x_1$ then $x_1$ is movable to a position $x''_1$ such that $x'_1$ and $x''_1$ are two opposite swap positions, then  $x''_1$ and $x_2$ are two opposite swap positions and, hence,
        $x'' = (x''_1,x_2) \in V'_{1,0}$, by definition.
        Recall that $x$ can also be moved to  $x' \in V'_{0,1}$. Then  $(\mathfrak{c}')$ holds for $x$.
    \item [$(b)$]  If $(\mathfrak{e})$ holds for $x_1$ then $x_1$ is movable to some $(0,0)$-position $x'''_1$ and to some $(1,1)$-position $x''''_1$.
    It is  not difficult to verify  that
    one of these two positions
    belongs to $V'_{0,0}$, while the other to $V'_{1,1}$ and, hence, $(\mathfrak{e'})$ holds for $x$.
    \end{itemize}
\item  Case $x' = (x'_1,x_2) \in V'_{1,0}$  is  
similar to the case $(1)$:
just swapping ``opposite" and ``parallel", as well as ``0,1" and ``1,0".
\item Case $x' = (x'_1,x_2) \in V'_{0,0}$. Consider the following three options for $x'$:
    \begin{enumerate}\itemsep0em
    \item If $(x'_1,x_2) \in [(0,0),(0,1)]$ then either $x_1$ is a (1,1)-position or $(\mathfrak{a_0})$ fails for $x_1$.
    Yet, the former case cannot occur as otherwise, $x = (x_1, x_2) \in V'_{1,1}$, giving a contradiction.
     In the latter case, either $(\mathfrak{c})$ or $(\mathfrak{e})$ holds for $x_1$, since  $G_1$ is tame.
     It is  easily seen that
     if $(\mathfrak{c})$ $($resp.,~$(\mathfrak{e}))$ holds for $x_1$ then $(\mathfrak{c}')$ $($resp.,~$(\mathfrak{e}'))$ holds for $x$.
    \item Case $x' = (x'_1,x_2) \in [(0,1),(0,0)]$  is similar to the case
    $(x'_1,x_2) \in [(0,0),(0,1)]$ treated in \rm{(a)}.
    \item If both $x'_1$ and $x_2$ are  $(n,n)$-positions, we consider two possibilities for $n$: $n$ is odd and $n$ is even. By checking carefully possible cases for $n$, one can verify that $(\mathfrak{a_0}')$, or $(\mathfrak{c}')$, or $(\mathfrak{e}')$ holds for $x$. We leave the checking task to the reader.
    \end{enumerate}
\end{enumerate}
By induction, we conclude that each position satisfies $(\mathfrak{a_0}')$, or $(\mathfrak{c}')$, or $(\mathfrak{e}')$
and, by Theorem \ref{tame.ss4}, sum  $G_1+G_2$  is tame.
Moreover, $V_{0,1} = V'_{0,1}$ and $V_{1,0} = V'_{1,0}$, implying that
$x = (x_1,x_2)$ is a swap position of the sum  $G_1+G_2$
if and only if
$x_i$  is a swap position of the summand  $G_i$  for  $i = 1,2$.
\end{proof}

\medskip

The following obvious generalization results from Theorems \ref{S.T} and \ref{S.SG}.

\begin{corollary}
If games  $G_1, \ldots, G_n$  are tame then
their sum  $G =  G_1 + \ldots + G_n$  is tame too. 
Moreover, a position  $x = (x^1,  \ldots, x^n)$  of  $G$
is a swap position of  $G$  if and only if  $x^i$  is a swap position of
$G_i$  for  $i \in  \{1, \ldots, n\}$.
Furthermore, $x$ is a $(1,0)$-position if and only if the number of $(1,0)$-positions
in the set $\{x^1, \ldots, x^n\}$  is odd.
\qed
\end{corollary}


\subsection{Sums of miserable, returnable, and forced games}

\begin{theorem} \label{S.M}
If games $G_1$ and $G_2$ are miserable then their sum $G_1 + G_2$  is miserable too.
Moreover, $x = (x_1, x_2)$  is a swap position of $G_1 + G_2$
if and only if each $x_i$  is a swap position of $G_i$  for  $i = 1,2$.
Furthermore, $x$ is a $(1,0)$-position of $G$ if and only
if either $x_1$ is a $(1,0)$-position in $G_1$ and $x_2$ is a $(0,1)$-position in $G_2$ or vise versa.
\end{theorem}

\begin{proof}
We proceed by induction on $d(x)$ and prove that every position $x$ in $G = G_1 + G_2$ satisfies condition
$(\mathfrak{a})$, or $(\mathfrak{b})$, or $(\mathfrak{c})$ of Definition \ref{D.WTMS}.
Note that $G_1$ and $G_2$ are tame, by Theorem \ref{MisT} and, hence,  $G$ is tame, by Theorem \ref{S.T}.

Let $x = (x_1,x_2)$ be a position of  $G$.
Clearly, $(\mathfrak{a})$ holds when  $d(x) = 0$,
since in this case both  $x_1$ and $x_2$ are terminal positions.

Assume that  $(\mathfrak{a})$, or $(\mathfrak{b})$, or $(\mathfrak{c})$
holds for every position $x$ with $d(x) \leq n$ for some $n \geq 1$.

Then, by induction, $x$ is  either a swap position or a $(k,k)$-position.
We prove that every position $x$ with $d(x) = n+1$ satisfies $(\mathfrak{a})$, or $(\mathfrak{b})$, or $(\mathfrak{c})$.
Assume that  $(\mathfrak{a})$ and $(\mathfrak{b})$  fail for $x$ and show that then $(\mathfrak{c})$ holds.

Indeed,  $x \notin V_{0,1} \cup V_{1,0}$, since $(\mathfrak{a})$ fails for  $x$, and $x \notin V_{0,0} \cup V_{1,1}$ since $(\mathfrak{b})$  fails for $x$. Therefore, $x$ is a $(m,m)$-position for some $m \geq 2$ since $G$ is tame.
Furthermore,  $x$ is movable to a swap position  $x'$, because $(\mathfrak{b})$ fails for  $x$.

Assume that  $x'$ is a $(0,1)$-position.

Furthermore, without loss of generality, we can assume that
move $x \rightarrow x'$ in  $G$  is realized by a move
$x_1 \rightarrow x_1'$  in  $G_1$.
Since $G$ is tame and $x' = (x_1', x_2)$ is a $(0,1)$-position, both
$x_1'$ and $x_2$ are swap positions, by Theorem \ref{S.T}.
Moreover, $\G(x_1') \oplus \G(x_2) = 0$  implies that
$\G(x_1') = \G(x_2)$  and that  $x_1'$ and $x_2$ are parallel.


In the case when  $x'$,  $x_1'$, and $x_2$  are
$(1,0)$-positions  rather than  $(0,1)$-positions,
similar arguments are applicable.

Since $G_1$ is miserable, $x_1$ satisfies $(\mathfrak{a})$, or  $(\mathfrak{b})$, or $(\mathfrak{c})$.

Since $x_1$ is movable to $x_1'$, which is a $(0,1)$-position, $(\mathfrak{b})$ fails for $x_1$.
We claim that $(\mathfrak{a})$  fails for $x_1$.
Indeed, otherwise  $x_1$ is a swap position. Note that $x_2$ is also a swap position and so $x$ is a swap position by Theorem \ref{S.T}. But this contradicts our assumption that $x$ is a $(m,m)$-position. Therefore $(\mathfrak{a})$  fails  and  $(\mathfrak{c})$ holds for $x_1$.

Then, there is also a move from $x_1$ to a  $(1,0)$-position  $x_1''$.
Note that  $x_1'$ and $x_1''$ are opposite while $x_1'$ and $x_2$ are parallel.
Hence, $x_1''$ and $x_2$ are opposite.
By Theorem \ref{S.T},  $x'' = (x_1'',x_2)$  is a swap position.
Moreover, it is a $(1,0)$-position and a option of  $x$.
Thus, $(\mathfrak{c})$ holds for $x$.

Then, by induction, $(\mathfrak{a})$, or $(\mathfrak{b})$, or $(\mathfrak{c})$  holds for every position.
Therefore, $G$ is miserable.
\end{proof}

\medskip

The following generalization results directly from Theorems \ref{S.M} and \ref{S.SG}.

\begin{corollary}
If games  $G_1, \ldots, G_n$  are miserable then their sum  $G =  G_1 + \ldots + G_n$  is miserable too.
Moreover, a position  $x = (x_1,  \ldots, x_n)$  of  $G$ is a swap position of  $G$  if and only if
$x_i$  is a swap position of $G_i$  for  $i \in \{1, \ldots, n\}$.
Furthermore, $x$ is a $(1,0)$-position if and only if
the number of $(1,0)$-positions in set $\{x_1, \ldots, x_n\}$  is odd.
\qed
\end{corollary}


The subclasses of forced or returnable miserable games are preserved under summation, as well.

\begin{proposition} \label{S.R}
The sum of miserable games is returnable whenever all summands are returnable.
\end{proposition}

\begin{proof}
It is sufficient to prove that $G_1+G_2$ is returnable whenever
$G_1$ and $G_2$ are miserable and returnable.
Let $x = (x_1,x_2)$ be a swap position in  $G$.
By Theorem \ref{S.M}, both $x_1$ and $x_2$ are swap positions.
Assume that $x$ is movable to some $x'$ in $G_1+G_2$.
Without loss of generality, assume that this move is realized by the move $x_1 \rightarrow x_1'$ in $G_1$.
Since $G_1$ is returnable, there exists a move $x_1' \rightarrow x_1''$ in $G_1$ such that
$x_1$ and $x_1''$ are either both $(0,1)$-positions or both $(1,0)$-positions.
Set $x' = (x_1',x_2)$ and $x'' = (x_1'',x_2)$
and consider moves $x \rightarrow x'$ and $x' \rightarrow x''$ in $G_1+G_2$.
By Theorems \ref{S.SG} and \ref{S.M},
$x$ and $x''$ are either both $(0,1)$-positions or both $(1,0)$-positions in $G$.
\end{proof}


\begin{proposition} \label{S.F}
The sum of miserable games is forced whenever all summands are forced.
\end{proposition}

\begin{proof}

It is sufficient to prove that $G = G_1+G_2$ is forced
whenever  $G_1$ and $G_2$ are miserable and forced.
Let $x = (x_1,x_2)$ be a swap position in $G$.
By Theorem \ref{S.M}, both $x_1$ and $x_2$ are swap positions.
If $x_1'$ is an  option of $x_1$ in $G_1$ then
$x_1'$ and  $x' = (x_1',x_2)$ are swap positions, by Theorem \ref{S.M}.
Moreover,  Theorems \ref{S.SG} and \ref{S.M} imply that
if $x$ is a $(0,1)$-position
(resp.,~$(1,0)$-position) then $x'$ is a $(1,0)$-position (resp.,~$(0,1)$-position).
These arguments are applicable to any option of $x$ in $G$.
\end{proof}


\section{Applications} \label{S.App}

In this section, we show that many classical games fall
into classes considered above.


\subsection{The game of {\sc Nim}}
\label{nim}
This game is played with $k$ piles of tokens.
By each move a player chooses one pile and removes an arbitrary
(positive) number of tokens from it.
The complete analysis of {\sc Nim} is was given
by Charles Bouton in \cite{Bou901}, who solved both the normal and mis\`ere versions.

Let us start with the trivial case  $k = 1$. The next statement is obvious.

\begin{lemma} \label{Nim-S}
One-pile {\sc Nim} is a strongly miserable game
with exactly one $(0,1)$-position, which is
the terminal position, and exactly one $(1,0)$-position, which is the single pile of size  $1$,
while the pile of size  $n$ is an $(n,n)$-position for all  $n \geq 2$.
\qed
\end{lemma}

Already the two-pile {\sc Nim} is not strongly miserable.
For example, {\sc Nim}$(2,2)$ is a $(0,0)$-position.


\begin{proposition} \label{Nim-M}
The game of {\sc Nim} is miserable and forced.
\end{proposition}

\begin{proof}
By Lemma \ref{Nim-S} and Theorem \ref{S.M}, 
{\sc Nim} is miserable.
Let us show that it is forced.
Let  $x = (x_1, \ldots, x_k)$  be a swap position.
By Theorem \ref{S.M}, each {\sc Nim}$(x_i)$  is a swap position,
implying either $x_i = 0$ or $x_i = 1$ for every $i$.
Obviously, every move from a swap positions
ends in another swap position and changes
the parity of the number of ones.
\end{proof}

The above arguments also prove that
the $(0,1)$-positions and $(1,0)$-positions alternate.
This immediately results in the following
characterization of the sets  $V_{0,1}$ and $V_{1,0}$.
\begin{proposition}
$V_{0,1} = \{\underset{2k \text{ entries $1$}}{(\underbrace{1,\ldots,1})} \mid
k \geq 0\}$ and $V_{1,0}= \{\underset{2k+1 \text{ entries $1$}}{(\underbrace{1,\ldots,1})}  \mid k \geq 0\}$.
\qed
\end{proposition}


\subsection{Subtraction games}\label{subtraction}
Subtraction game, denoted by $\S(X)$, is played with
a finite pile of tokens and a set  $X$  of positive integers,
which may be finite or infinite.
A move is to choose an element of $X$ and remove this number of tokens from the pile.
Various aspects of this game are exposed in \cite{ANW07, AB95, BCG01-04, CH10, Fer74}.

In \cite{Fer74}, Ferguson shows that in any subtraction game
each non-terminal $0$-position is movable to a $1$-position.
This and Theorem \ref{SM=P} imply the following statement.

\begin{proposition} \label{sub-sm}
Subtraction games are strongly miserable.
\qed
\end{proposition}

Since the proof by Ferguson \cite{Fer74} is
very short and elegant, we copy it here for the reader's convenience.

\begin{proposition} [\cite{Fer74}] \label{p4}
Every subtraction game satisfies  
property \rm{(v)} of Theorem \ref{SM=P}.
\end{proposition}

Proposition \ref{p4} is based on the following lemma.

\begin{lemma} [\cite{Fer74}] \label{L.F}
Set  $k = \min(X)$. Then $\G(x) = 0$  if and only if $\G(x+k) = 1$.
\end{lemma}

\begin{proof}
Since $k \in X$,  $\G(x) = 0$  implies $\G(x+k) \neq 0$ for all $x$.

For the necessary condition, assume for contradiction that there exists the smallest $x$ such that  $\G(x) = 0$  and $\G(x+k) > 1$. By the definition of SG values, there exists $s \in X$  such that $\G(x+k-s)=1$. Since $k = \min(X)$,  $k - s \leq 0$. Moreover, $x+k-s \geq k$ or $x - s \geq 0$ (otherwise, there is no move from $x+k-s$ while $\G(x+k-s) = 1$). Furthermore,  $\G(x) = 0$ implies  $\G(x - s) > 0$. Thus there exists $s' \in X$  such that  $\G(x - s - s') = 0$ by the definition of SG values.

Let $y = x - s - s'$. Then $y < x$ and $\G(y) = 0$, implying that $\G(y+k) = 1$, by the choice of the smallest $x$. However, the last equation implies that  $\G(y+k+s') \neq 1$ or, equivalently, $\G(x-s+k) \neq 1$, contradicting $\G(x+k-s)=1$ as above.

Conversely, if $\G(x) = 1$ and $\G(x - k) \neq 0$, there exists $s \in X$ such that  $\G(x-k-s) = 0$. By the necessary condition, $\G(x-s) = 1$, which contradicts $\G(x) = 1$.
\end{proof}


\begin{proof} [Proof of Proposition \ref{p4}]
Given any non-terminal  $x$  such that $\G(x) = 0$, one has $\G(x - k) \neq 0$, where $k$  is the smallest element of  $X$. This implies that there is an  $s \in X$ such that $\G(x - k - s) = 0$. From Lemma \ref{L.F},  $\G(x - s) = 1$.
\end{proof}


\subsection{Game {\sc Mark}}\label{mark}

A game played with a single pile is called a single-pile {\sc Nim}-like game if
two players take turns removing tokens from that pile. 
After Subsections \ref{nim} and \ref{subtraction},
one may ask whether each single-pile {\sc Nim}-like game is strongly miserable.
The is not the case.
Moreover, such a game may be not even domestic.
For example, let us consider the following
single-pile {\sc Nim}-like game suggested by  Fraenkel \cite{Fra11}
and called {\sc Mark}.
By one move a pile of size $n$  should be reduced to
either $n-1$  or $\lfloor \frac{n}{2} \rfloor$.

\begin{proposition} \label{Mark-nD}
Game {\sc Mark} is not domestic.
\end{proposition}

\proof
It is not difficult to verify that  8 is a (0,2)-position.
\qed


\subsection{Game {\sc Euclid}} \label{euclid}

In 1969  Cole and Davie \cite{CD69} introduced game {\sc Euclid}.
It is played with two piles of tokens.
By one move a player has to remove from the greater pile
any number of tokens that is an integer multiple of the size of the smaller pile.
The game ends when one of the piles is empty.
A position of two piles of sizes $x$ and $y$ is denoted by $(x,y)$.
It was shown in  \cite{CD69}  that
$(x,y)$ is a $\P$-position if and only if
$x < y < \phi x$, where $\phi = (1+\sqrt{5})/2$ is the golden ratio \cite{CD69}.

In 1997, Grossman \cite{Gro97} proposed a modification
of this game in which the entries must stay positive.
In particular, move  $(x,y) \rightarrow (x,0)$  is not allowed
even if $y$ is a multiple  of  $x$.
Thus, the terminal positions of this game are  $(x,x)$
for some positive $x$.

Note that Grossman's variant is not the mis\`{e}re version of  {\sc Euclid} by Cole and Davie.
Also note that in the literature the examples \cite{Gur07, Hof08, Len03, Niv06} referred to as {\sc Euclid} are Grossman's
version, not Cole and Davie's version.

The SG function of Grossman's variant was solved in \cite{Niv06} and that
of the original game {\sc Euclid} was solved later in \cite{CHL11}, where it was shown that
these two SG functions are very similar.
Some other variants were also studied in \cite{CH-euclid, Col08, Ho11}.

We now analyze miserability of these two games.
Miserability of Grossman's variant was analyzed in \cite{Gur07}.


\begin{proposition} \label{Euclid-M}
Both Cole and Davie's game and Grossman's game of {\sc Euclid} are miserable and forced.
\end{proposition}
\begin{proof}
We first prove that Cole and Davie's game miserable.
Set $V'_{0,1} = \{(0,x), (x,0) \mid x \in \Z_{> 0}\}$ and
$V'_{1,0} = \{(x,x) \mid x \in \Z_{> 0}\}$ in which $\Z_{> 0}$
is the set of positive integers.
Note that if $v \in V'_{0,1}$, then $v$ is a terminal and, hence,  a $(0,1)$-position.
If $v \in V'_{1,0}$ then $v$ is movable to a terminal position and, moveover,
this is the only move available from $v$; hence, $v$ is a $(1,0)$-position.

It is easily  seen that
if $v \notin V'_{0,1} \cup V'_{1,0}$ then either $v$ is not movable to $V'_{0,1} \cup V'_{1,0}$ or
$v$ is movable to $V'_{0,1}$ and to $V'_{1,0}$.
Then, by Theorem \ref{Mis.ss4}, the game 
is miserable and, moreover,
$V_{0,1} = V'_{0,1}$ and $V_{1,0} = V'_{1,0}$.
It follows also that this game is forced.

For Grossman's game, we set
$V'_{0,1} = \{(x,x) \mid x \in \Z_{> 0}\}$ and
$V'_{1,0} = \{(x,2x), (2x,x) \mid x \in \Z_{> 0}\}$
and the same arguments work.
\end{proof}


\subsection{Game {\sc Wythoff}}

The {\sc Wythoff} game \cite{Wyt907} is a modification of the two-pile {\sc Nim}
in which a player by one move is allowed to remove  either

 \medskip

(i) an arbitrary number of tokens from one pile, or

\smallskip

(ii) the same number of tokens from both.

\medskip

Two piles of sizes $x$ and $y$ define a position  $(x,y)$.
By symmetry, $(x,y)$  and  $(y,x)$  are equivalent;
we will assume that $x \leq y$
unless the converse is explicitly said. 

Let $(x_n, y_n)_{n \geq 0}$, where
$x_i < x_j$ if $i < j$, be the sequence of $\P$-positions of the game. 
Wythoff \cite{Wyt907} proved that $(x_n, y_n)$ is a $\P$-position if and only if
$x_n = \lfloor \phi n \rfloor$ and $y_n = \lfloor \phi^2 n \rfloor$, where
$\phi = (1+\sqrt{5})/2$ is the golden ratio.
Note that $\lfloor \phi^2 n \rfloor =\lfloor \phi n \rfloor + n$.

The game {\sc Wythoff} and numerous modifications of it were studied intensively in the literature:
\cite{BF90, BGO11, DFNR10, Fra82, Fra98, Fra84, Gur12, Ho12, Urb11}.
However, no explicit formula is known for the SG function of this game.
In \cite{Fra84}, Fraenkel analyzed the mis\`{e}re version of {\sc Wythoff} and
characterized its $\P$-positions.
Interestingly, the $\P$-positions of the normal and mis\`ere versions of {\sc Wythoff}
differ only by six positions:
$\{(0,0), (1,2), (2,1)\} \in V_P^N \setminus V_P^M$, while
$\{(0,1), (1,0), (2,2)\} \in V_P^M \setminus V_P^N$.
Here $V_P^N$ (resp.,~$V_P^M$) is the set $\P$-positions in the normal (resp.,~mis\`{e}re) version.
One can check this claim by comparing \cite[Proposition 2]{DFNR10} and  \cite[Theorem 2.1]{Fra84}.
Using these results, one can verify directly that the game {\sc Wythoff} is miserable.
Here, we provide an alternative proof using Theorem \ref{Mis.ss4}.


\begin{proposition} \label{Wyt-M}
Game {\sc Wythoff} is miserable.
\end{proposition}

\begin{proof}
 Let us set   $V'_{0,1} = \{(0,0),(1,2), (2,1)\}$ and  $V'_{1,0} = \{(0,1),(1,0),(2,2)\}$.
  One can easily verify the containments
  $V'_{0,1} \subseteq V_{0,1}$ and $V'_{1,0} \subseteq V_{1,0}$.
  Let $(x,y)$ be a position that does not belong to  $V'_{0,1} \cup V'_{1,0}$.
  It is easily seen that
  either $(x,y)$ is not movable to $V'_{0,1} \cup V'_{1,0}$ or
  $(x,y)$ is movable to both $V'_{0,1}$ and $V'_{1,0}$.
  Thus, by Theorem \ref{Mis.ss4}, the game {\sc Wythoff} is miserable and, moreover,
  $V_{0,1} = V'_{0,1}$ and $V_{1,0} = V'_{1,0}$.
\end{proof}

Note that $(3,5)$ is a $(0,0)$-position and, thus, {\sc Wythoff} is not strongly miserable.


\begin{proposition}
The game {\sc Wythoff} is returnable but not forced.
\end{proposition}
\begin{proof}
 There is a move from (2, 2), which is a $(1, 0)$-position, to (1, 1), which is a (2, 2)-position; hence, the game 
 is not forced. It is easily seen that the game is returnable.
\end{proof}


\subsection{{Game \sc Wyt}$(a)$}  \label{SS.WYT(a)}

In \cite{Fra82} Fraenkel, for any positive integer $a$, introduced the following
generalization  {\sc Wyt}$(a)$  of the game {\sc Wythoff}.
This game is also played with two piles of tokens and by one move a player is allowed

\medskip

(i) to remove an arbitrary number of tokens from one pile, or

\smallskip

(ii) to remove $k$ tokens from one pile and $l$ tokens from the other pile such that $|k-l|<a$.

\medskip

The game {\sc Wyt}$(a)$ was studied by Fraenkel  \cite{Fra82, Fra84}.
Note that {\sc Wyt}$(1)$ is {\sc Wythoff} and,  hence, it is miserable.

\begin{proposition} \label{Fr-Nim-SM}
 Game {\sc Wyt}$(a)$ is strongly miserable whenever $a \geq 2$.
\end{proposition}

We first recall results on $\P$-positions of the normal and mis\`{e}re versions.


\begin{proposition} [\cite{Fra82}] \label{Fr-Nim-P}   
For $a \geq 2$, the sequence $(x_n,y_n)_{n \geq 0}$ of $\P$-positions of {\sc Wyt}$(a)$ satisfies the following conditions:
\begin{itemize}  \itemsep0em
\item [\rm{(i)}] $(x_0,y_0) = (0,0)$;
\item [\rm{(ii)}] for $n \geq 1$, $x_n = \operatorname{mex} \{x_i, y_i  \mid 0 \leq i < n\}$ and $y_n = x_n + an$.
\end{itemize}
\end{proposition}


\begin{proposition} [\cite{Fra84}] \label{Fr-Nim-P-mis}    
For $a \geq 2$, the sequence $(x'_n,y'_n)_{n \geq 0}$ of $\P$-positions of mis\`{e}re {\sc Wyt}$(a)$ satisfies the following conditions:
\begin{itemize}  \itemsep0em
\item [\rm{(i)}] $(x'_0,y'_0) = (0,1)$;
\item [\rm{(ii)}] for $n \geq 1$, $x'_n = \operatorname{mex} \{x'_i, y'_i  \mid 0 \leq i < n\}$ and $y'_n = x'_n + an + 1$.
\end{itemize}
\end{proposition}


\begin{corollary} \label{Fr-Nim-P-dif}
For $a \geq 2$, two sets of $\P$-positions of {\sc Wyt}$(a)$ and its mis\`{e}re version are disjoint.
\end{corollary}
\begin{proof}
Let $(x_m,y_m)$ be a $\P$-position of {\sc Wyt}$(a)$ and
let $(x_n',y_n')$ be a $\P$-position of mis\`{e}re {\sc Wyt}$(a)$.
If these two positions are coincident then $x_m = x_n'$ and $x_m + am = x'_n + an + 1$.
One can then simplify to obtain the equation $a(m - n) = 1$, giving a contradiction as
$1$ cannot be multiple of $a$.
\end{proof}

\begin{proof}[Proof of Proposition \ref{Fr-Nim-SM}]
It follows immediately from Corollary  \ref{Fr-Nim-P-dif} and Theorem \ref{SM=P} \rm{(iii)}.
\end{proof}


\subsection{Game {\sc Wyt}$(a,b)$}  \label{SS.WYT(a,b)}

Game {\sc Wyt}$(a,b)$ was introduced in \cite{Gur12},
for any two non-negative integers $a$ and $b$,  as follows.
Like {\sc Wythoff}, it is played with two piles of tokens.
By one move a player is allowed to delete $x \geq 0$ tokens from one pile and
$y \geq 0$ tokens from the other such that $x+y >0$ and 
($\operatorname{min}(x,y) < b$ or $|x-y| < a$).
Note that
{\sc Wyt}$(0,1)$ is the two-pile {\sc Nim},
{\sc Wyt}$(1,1)$ is {\sc Wythoff}, and
{\sc Wyt}$(a,1)$ is {\sc Wyt}$(a)$.

\smallskip

The following recursive solution
of the normal and mis\`{e}re versions of the game
was given in \cite{Gur12}

Given an integer $b \geq 1$ and a finite set $S$ of $m$
non-negative integers $s_1, \ldots, s_m$ such that $s_1  < \dots < s_m$,
let us set $s_0 = -b$ and $s_{m+1} = +\infty$.
Then, there exists the smallest index $i \in \{0, 1, \ldots , m\}$ such that $s_{i+1} - s_i > b$.
Let us define a function  $\operatorname{mex_b}$ of $S$ as follows:
\[\operatorname{mex_b}(S) = s_i + b \]
It is easily seen that $\operatorname{mex_b}(\emptyset) = 0$ and that
$\operatorname{mex_b(S)}$ equals $\operatorname{mex}(S)$  when $b=1$, that is,
$\operatorname{mex_1} = \operatorname{mex}$.

\smallskip

The $\P$-positions of the normal and its mis\`{e}re versions of game {\sc Wyt}$(a,b)$
are characterized in \cite{Gur12} as follows.

\begin{proposition} [\cite{Gur12}] \label{Gur-Wyt-P}
The sequence $(x_n,y_n)_{n \geq 0}$ of the $\P$-positions of the normal version of game {\sc Wyt}$(a,b)$
satisfies the following recursion:
\[ x_n = \operatorname{mex_b} \{x_i, y_i \mid 0 \leq i < n\}, \quad y_n = x_n + a n. \]
\qed
\end{proposition}

\begin{proposition} [\cite{Gur12}] \label{Gur-Wyt-P-mis}
The sequence $(x'_n,y'_n)_{n \geq 0}$ of the $\P$-positions of mis\`{e}re version
of game {\sc Wyt}$(a,b)$ satisfies the following recursion:
\begin{itemize}  \itemsep0em
\item [\rm{(i)}] if $a = 1$, then $(x_0', y_0') = (b+1,b+1)$ and $x_n' =
\operatorname{mex_b} \{x_i', y_i' \mid 0 \leq i < n\}, \quad y_n' = x_n' + a n$;
\item [\rm{(ii)}] if $a \geq 2$, then $x_n' =
\operatorname{mex_b} \{x_i', y_i' \mid 0 \leq i < n\}, \quad y_n' = x_n' + a n + 1$.
\qed
\end{itemize}
\end{proposition}

\begin{proposition} [\cite{Gur12}] \label{tNIMab}
Game {\sc Wyt}$(a,b)$ is strongly miserable whenever $a \geq 2$.
\qed
\end{proposition}


\begin{proof}
We only need to show that
the normal and mis\`ere versions of {\sc Wyt}$(a,b)$
do not share $\P$-positions,
or in other words, that there is no $(0,0)$-position.
Then,  game {\sc Wyt}$(a,b)$ is strongly miserable, by Theorem \ref{SM=P}.

Let $(x_n,y_n)$  and   and $(x_m',y_m')$
be $\P$-positions of the normal and mis\`ere versions of {\sc Wyt}$(a,b)$, respectively.
 Suppose these two positions coincide, $x_m' = x_m$ and $y_m' = y_n$.
 By Propositions \ref{Gur-Wyt-P} and \ref{Gur-Wyt-P-mis} for case $a \geq 2$, one obtains equality $a(n - m) = 1$,
 which is a contradiction since 1 cannot be a multiple of $a$.
\end{proof}


The case  $a \leq 1$  was studied in \cite{Gur11, Gur12}.
Combining these results with Proposition \ref{tNIMab} we
obtain the following criterion.

\begin{proposition}
Game {\sc Wyt}$(a,b)$ is miserable and returnable if
$(a = 1$ and $b \geq 1)$ or  $(b = 1$ and $a \leq 1)$.
Otherwise, the game is strongly miserable.
\qed
\end{proposition}


\subsection{Moore's {\sc Nim$_{n,\leq k}$} and  its variants} \label{SS.NIMB}


\subsubsection{Moore's {\sc Nim$_{n,\leq k}$}}
The following game was introduced in 1910  by Moore \cite{Moore}.
Let $k$ and $n$ be two positive integers such that $k \leq n$.
By one move a player has to reduce (strictly) at least 1  and at most  $k$
from given  $n$  piles of  $(x_1, \ldots, x_n)$  tokens.
Moore denoted this game {\sc Nim}$_k$, but
we use notation {\sc Nim}$_{n,\leq k}$ to include $n$.

We will show that game of {\sc Nim}$_{n,\leq k}$ is miserable. For $k=1$, it is known.

\begin{proposition} \label{NIMB-M}
The game of {\sc Nim}$_{n,\leq k}$ is miserable for $2 \leq k < n$.
Moreover, let $x = (x_1,\ldots,x_n)$ be a position in {\sc Nim}$_{n,\leq k}$ and
$l$ be the number of non-empty piles in $x$. Then
\begin{itemize}  \itemsep0em
\item [$(a)$] $x$ is a $(0,1)$-position if and only if $x_i \leq 1$ for all $i$ and $l \equiv 0 \bmod{(k+1)}$;
\item [$(b)$] $x$ is a $(1,0)$-position if and only if $x_i \leq 1$ for all $i$ and $l \equiv 1 \bmod{(k+1)}$.
\end{itemize}
\end{proposition}

\begin{proof}
Let us set
\[V'_{0,1} = \{(x_1,\ldots,x_n) \mid \text{ $x_i \leq 1$ for all $i$ and $l \equiv 0 \bmod{(k+1)}$}\};\]
\[V'_{1,0} = \{(x_1,\ldots,x_n) \mid \text{ $x_i \leq 1$ for all $i$ and $l \equiv 1 \bmod{(k+1)}$}\}.\]

We verify the conditions \rm{(i)} - \rm{M(v)} of Theorem \ref{Mis.ss4}.
Condition \rm{(i)} holds since there is no move between two arbitrary positions in each set since such a move must reduce $k+1$ piles.
Condition \rm{(ii)} holds sine $V'_{0,1}$ contains the terminal position $(0,0,\ldots,0)$.
Condition \rm{(iii)} holds since from every non-terminal position
in $V'_{0,1}$, the move removing exactly $k$ tokens terminates in $V'_{1,0}$.
Condition \rm{(iv)} holds since from every position in $V'_{1,0}$, the move removing exactly one token terminates in $V'_{0,1}$.
It remains to verify condition \rm{M(v)}.

Let $x$ be a position not in the set $V'_{0,1} \cup V'_{1,0}$. If there is no move from $x$ that terminates in  $V'_{0,1} \cup V'_{1,0}$ then the condition \rm{M(v)} holds and we are done. Assume that this is not the case. Then there exists one move $M_1$ from $x$ that terminates in either $V'_{0,1}$ or $V'_{1,0}$. We need to prove that $x$ is movable to both $V'_{0,1}$ and $V'_{1,0}$

Note that a move from $x$ reduces at most $k$ piles $x_{\pi(1)},x_{\pi(2)},\ldots,x_{\pi(k)}$ for a permutation $\pi$, meaning
    \[(M_1): \quad (x_{\pi(1)},x_{\pi(2)},\ldots,x_{\pi(k)}) \rightarrow (x_{\pi(1)}-y_1,x_{\pi(2)}-y_2,\ldots,x_{\pi(k)}-y_k)\]
with at least some $y_j \geq 1$.

\begin{enumerate} \itemsep0em
\item If the move $(M_1)$ terminates in $V'_{0,1}$, then it leaves $m(k+1)$ entries of size 1.
    \begin{enumerate} \itemsep0em
    \item If  $x_{\pi(i)}-y_i = 1$ for all $i$, then the corresponding move
    \[(M_2): \quad (x_{\pi(1)},x_{\pi(2)},\ldots,x_{\pi(k)}) \rightarrow (x_{\pi(1)}-y_1-1,x_{\pi(2)}-y_2-1,\ldots,x_{\pi(k)}-y_k-1)\]
    terminates in $V'_{1,0}$, leaving $(m-1)(k+1)+1$ entries of size 1.
    \item If $x_{\pi(i)}-y_i = 0$ for some $i$, then either there exists $i_0$
    such that $y_{i_0} \geq 2$ or there exist $i_0$ and $j_0$ such that $y_{i_0} \geq 1$ and $y_{j_0} \geq 1$.
    In fact, if otherwise, $x \in V'_{1,0}$, giving a contradiction. In either of cases, we can choose $(y'_1, y'_2, \ldots, y'_k)$ such that $0 \leq y'_i \leq y_i$ and $y'_1+y'_2+ \cdots + y'_k = y_1+y_2 + \cdots + y_k -1$. Then the corresponding move
        \[(M_3) \quad (x_{\pi(1)},x_{\pi(2)},\ldots,x_{\pi(k)}) \rightarrow (x_{\pi(1)}-y'_1,x_{\pi(2)}-y'_2,\ldots,x_{\pi(k)}-y'_k)\]
        terminates in $V'_{1,0}$, leaving $m(k+1)+1$ entries of size 1.
    \end{enumerate}
\item If the move $(M_1)$ terminates in $V'_{1,0}$, then it leaves $m(k+1)+1$ entries of size 1.
    \begin{enumerate} \itemsep0em
    \item If $x_{\pi(i_0)} > y_{i_0}$ for some $i_0$, then we define $y'_i = y_i$ for all $i$, except for $y'_{i_0} = x_{\pi(i_0)}$. The move
        \[(M_4) \quad (x_{\pi(1)},x_{\pi(2)},\ldots,x_{\pi(k)}) \rightarrow (x_{\pi(1)}-y'_1,x_{\pi(2)}-y'_2,\ldots,x_{\pi(k)}-y'_k)\]
        terminates in $V'_{0,1}$, leaves $m(k+1)$ entries of size 1. Here $(M_4)$ imitates $(M_1)$ before removing the whole pile $x_{\pi(i_0)}$.
    \item If $x_{\pi(i)} = y_i$ for all $i$, we consider two cases.
        \begin{enumerate}
        \item If $y_{i_0} = 0$ for some $i_0$, we can choose some pile $x_{j_0} \notin \{x_{\pi(1)},x_{\pi(2)},\ldots,x_{\pi(k)}\}$ of size 1 which is not touched in the move $(M_1)$. Then the move
            \[(M_5): \quad (x_{j_0}, x_{\pi(1)},x_{\pi(2)},\ldots,x_{\pi(k)}) \rightarrow (0, x_{\pi(1)}-y_1,x_{\pi(2)}-y_2,\ldots,x_{\pi(k)}-y_k)\]
            terminates in $V'_{0,1}$. Note that  $(M_5)$ imitates $(M_1)$ before removing the pile $x_{j_0}$, resulting in $m(k+1)$ entries of size 1.
        \item If $x_{\pi(i)} = y_i > 0$ for all $i$, then there exists $i_0$ such that $x_{i_0} \geq 2$. Otherwise, $x \in V'_{0,1}$. Now, we have $y_{i_0} -1 = x_{i_0} - 1 \geq 1$. Then the move
            \begin{align*}
            (M_6): & \quad (x_{\pi(1)},x_{\pi(2)},\ldots,x_{\pi(k)}) \\
                   & \rightarrow (x_{\pi(1)}-(y_1-1),x_{\pi(2)}-(y_2-1),\ldots,x_{\pi(k)}-(y_k-1))
            \end{align*}
            terminates in $V'_{0,1}$, leaving $(m+1)(k+1)$ entries of size 1.
        \end{enumerate}
    \end{enumerate}
\end{enumerate}
\end{proof}


\subsubsection{An extension of {\sc Nim}$_{n,\leq k}$}
We extend {\sc Nim}$_{n,\leq k}$ to a game called {\sc Extended Nim}$_{n,\leq k}$
that has an extra pile with  $x_0$ tokens.
By one move, it is allowed to
reduce  $x_0$  and at most  $k$  of the remaining  $n$  piles.
Note that at least one pile must be reduced strictly;
reducing  $x_0$ is not compulsory and reducing only $x_0$ is legal.
When $k = n-1$, the game {\sc Extended Nim}$_{n,\leq n-1}$ is called
{\sc Extended Complementary Nim, or Exco-Nim}, for short, \cite{BGHM15}.

\begin{proposition} \label{NIMBE-M}
Let $n \geq 3$  and  $1 \leq k < n$.
Game {\sc Extended Nim}$_{n,\leq k}$ is miserable.  Moreover,
$x = (x_0, x_1, \ldots, x_n)$  is a $(0,1)$- $($resp.~$(1,0))$-position
if and only if  $x_0 = 0$  and   $(x_1, \ldots, x_n)$ is
a $(0,1)$- $($resp.~$(1,0))$-position of {\sc Nim}$_{n,\leq k}$.
\qed
\end{proposition}

The proof is essentially similar to that of Proposition \ref{NIMB-M} and we leave it to the reader.


\subsubsection{{\sc Exact $k$-Nim}}
Let us consider a modification of {\sc Nim}$_{n,\leq k}$ in which by one move a player must (strictly) reduce exactly $k$ piles.
We denote this game by  {\sc Nim}$_{n, = k}$.
A closed formula for its SG function was obtained in \cite{BGHMM15}
for the case  $n \leq 2k$.

We prove that the game is miserable when  $n \leq 2k$.
We start with the case $n = 2k$.

\begin{proposition}
Game {\sc Nim}$_{2k, = k}$ is miserable.
Moreover, $x = (x _1, \ldots, x_n)$ is a
\begin{itemize} \itemsep0em
\item  [\rm{(i)}] $(0,1)$-position if and only if $x_1  = \cdots = x_{k+1} \leq 1$;
\item  [\rm{(ii)}] $(1,0)$-position if and only if $d(x) = 1$.
\end{itemize}
\end{proposition}
Recall that $d(x)$ denotes the greatest number of successive moves from  $x$  to the terminal position.
\begin{proof}
We leave to the reader to check that two sets
 $V'_{0,1} = \{x = (x_1, \ldots, x_n) \mid x_1 =  \cdots = x_{k+1} \leq 1\}$ and
 $V'_{1,0} = \{x \mid d(x) = 1\}$ satisfy conditions in Theorem \ref{Mis.ss4};
 hence the game is miserable with
 $V_{0,1} = V'_{0,1}$ and $V_{1,0} = V'_{1,0}$.
\end{proof}

Recall that $d(x)$ is the largest number of moves from $x$ to the terminal position.

\begin{proposition}
Game {\sc Nim}$_{n, = k}$ with $n < 2k$ is strongly miserable.
\end{proposition}

\begin{proof}
Note that if $x = (x_1, \ldots, x_n)$ is a $\P$-position then $x$ is terminal.
Indeed, every non-terminal position is movable to the terminal position by eliminating
the piles $x_1, \ldots, x_k$ and, thus, leaving at most $n-k = k-1$ nonempty piles.
By definition, a positions with at most $k-1$ nonempty piles is terminal.

In other words, $\G(x) = 0$ if and only if $x$ is the terminal position,
which is also a $(0,1)$-position.
In particular, there are no $(0,0)$-position
and,  by Theorem \ref{SM=P}, the game is strongly miserable.
\end{proof}

Many games {\sc Nim}$_{n, = k}$ with $k < n/2$ are not even domestic.
For example, our computations show that
$(1,2,3,3,3)$  is a $(0,2)$-position of {\sc Nim}$_{5, = 2}$.


\subsubsection{{\sc Slow $k$-Nim}}
Let us now consider a modification of {\sc Nim}$_{n, \leq k}$ in which
a move consists of choosing at least one and at most $k$ from $n$ piles
and removing {\em exactly one} token from each of them.
The obtained game is denoted by  {\sc Nim}$^1_{n, \leq k}$;
it was analyzed in \cite{GurHo15-slow}.

Relations between the normal and mis\`ere versions are summarized by the
following statement.

\begin{proposition}
For $k \geq n-1$, the game of {\sc Slow $k$-Nim} is miserable. Moreover,
\begin{enumerate} \itemsep0em
\item [\rm{(i)}] if $k = n$, $V_{0,1} = \{(0,0,\ldots,0,2j) \mid j \in \Z_{\geq}\}$ and $V_{1,0} = \{(0,0,\ldots,0,2j+1) \mid j \in \Z_{\geq}\}$;
\item [\rm{(ii)}] if $k = n-1$, $V_{0,1} = \{(i,i,\ldots,i,i+2j) \mid i, j \in \Z_{\geq}\}$ and $V_{1,0} = \{(i,i,\ldots,i,i+2j+1) \mid i, j \in \Z_{\geq}\}$.
\end{enumerate}
\end{proposition}
\begin{proof}
For $k = n$ and For $k = n-1$ let us respectively set

\medskip

$V'_{0,1} = \{(0,0,\ldots,0,2j) \mid j \in \Z_{\geq}\}$ and
$V'_{1,0} = \{(0,0,\ldots,0,2j+1) \mid j \in \Z_{\geq}\}$.

\medskip

$V'_{0,1} = \{(i,i,\ldots,i,i+2j) \mid i, j \in \Z_{\geq}\}$ and
$V'_{1,0} = \{(i,i,\ldots,i,i+2j+1) \mid  i, j \in \Z_{\geq}\}$.

\medskip

We leave to the reader to verify that
these two sets $V'_{0,1}$ and $V'_{1,0}$ satisfy all conditions of Theorem \ref{Mis.ss4}
and, hence, the game is miserable with $V_{0,1} = V'_{0,1}$ and $V_{1,0} = V'_{1,0}$.
\end{proof}

Our computations show that game  {\sc Nim}$^1_{4, \leq 2}$
is not domestic; for example,  $(1,1,2,3)$ is a $(4,0)$-positions.
Thus, case  $k = n - 2$  differs a lot from the case $k = n - 1$
corresponding to the Complementary {\sc Nim}.


\subsection{Heap overlapping {\sc Nim}}
\label{SS-HO-NIM}
The following generalization of {\sc Nim}
was introduced in \cite{BGHM15} and called {\sc HO-Nim},
where HO stands for ``Heap Overlapping".
Given a ground set $V$, a position of this game involves
a family of its subsets $\cH = \{H_1, \ldots, H_n\}$.
Furthermore, a move from this position consists of choosing
a non-empty subset  $S$  of some set  $H_i$,
deleting  $S \cap H_j$  from each  $H_j$, and getting thus
a new position  $\{H_1 \setminus S, \ldots, H_n \setminus S\}$.
Note that {\sc HO-Nim} ($\cH$) is the classic {\sc Nim}
whenever the subsets $H_i$ are pairwise disjoint.

In this subsection we construct examples
of domestic but not tame {\sc HO-Nim} games.


\begin{definition} \label{Cn}
Given a ground set  $V$  partitioned by $n \geq 3$
pairwise disjoint subsets $V_1, \ldots, V_n$, let us set
$H_1 = V_1 \cup V_2, 
H_2 = V_2 \cup V_3, 
\ldots,
H_{n-1} = V_{n-1} \cup V_{n}, 
H_n = V_n \cup V_1,$
and $\cH = \{H_1, H_2, \ldots, H_n\}$.
We denote the corresponding position by
$(|V_1|, |V_2|, \ldots, |V_n|)$ and game by $\cH(C_n)$. \qed
\end{definition}


\begin{proposition}
{\sc HO-Nim} $\cH(C_4)$ is miserable and forced.
{\sc HO-Nim} $\cH(C_5)$ is domestic but not tame.
{\sc HO-Nim} $\cH(C_6)$ is not domestic.
\end{proposition}

\begin{proof}
By symmetry, the positions
$(x_1, x_2, \ldots, x_n)$ and  $(x_2, x_3, \ldots, x_n, x_1)$ are equivalent.
We denote by $[(x_1, x_2, \ldots, x_n)]$  the set of positions
equivalent with  $(x_1, x_2, \ldots, x_n)$.

For $\cH(C_4)$, set $V'_{0,1} = \{(0,0,0,0)\} \cup [(0,1,0,1)]$ and  $V'_{1,0} = [(0,0,0,1)]$.
By Theorem \ref{Mis.ss4}, the game is miserable; moreover, $V'_{0,1} = V_{0,1}$ and $V'_{1,0} = V_{1,0}$.
Furthermore, every move from a position in $V_{1,0}$ ends in $(0,0,0,0)$, which is
the (unique) terminal, while every move from a position of $V_{0,1}$ terminates in $V_{1,0}$.
Hence, the game is forced.
Note that the $(0,0)$-positions of this game are
$\{(a,b,a,b) \mid a,b \in \Z_{\geq 0}, a+b \geq 2\}$.

For $\cH(C_5)$, direct computation shows that $x = (2,0,1,1,1)$  is a $(5,1)$-position.
Therefore, {\sc HO-Nim} $\cH(C_5)$ is not tame.
Let us show that $\cH(C_5)$ is domestic.

It can be easily verified that the set of $(0,0)$-positions is
\begin{align*}\label{eq-C5}
V_{0,0} =& [a, c+a, b+a, a, c+b+a] \cup [a, c+a, a, b+a, c+b+a] \\
         &\cup [a, a, c+a, b+a, c+b+a] \quad \text{ with } c,b,a \in \Z_{\geq 0}. 
\end{align*}
Let us set \begin{align*}
V'_{0,0} & = V_{0,0}, \\
V'_{0,1} & = \{(0,0,0,0,0), (1,1,1,1,1)\} \cup [(0,0,1,0,1)], \\
V'_{1,0} &= [(0,0,0,0,1)] \cup [(0,1,1,1,1)].
\end{align*}
It is easily seen that three sets
$V'_{0,0}, V'_{0,1}$, and $V'_{1,0}$
satisfy conditions of Theorem \ref{Dos.ss4} and, thus, the game is domestic.

Game $\cH(C_6)$ is not domestic, since $(1,1,1,1,1,1)$ is a $(0,2)$-position in it.
\end{proof}

\begin{definition} \label{Kn}
Given a ground set  $V$  partitioned by $n \geq 3$ pairwise disjoint subsets $V_1, \ldots, V_n$, let us set
$H_i = V_{i} \cup V_{i+1}$ for $1 \leq i \leq n-1$ and
$\cH = \{H_i \mid 1 \leq i \leq n-1\}$.
We denote the corresponding position by
$(|V_1|, |V_2|, \ldots, |V_n|)$  and
the game by {\sc Ho-Nim} $\cH(P_n)$. \qed
\end{definition}

\begin{proposition}
{\sc Ho-Nim} $\cH(P_3)$ is miserable.
{\sc Ho-Nim} $\cH(P_4)$ and $\cH(P_5)$ are domestic but not tame.
{\sc Ho-Nim} $\cH(P_6)$  is not domestic.
\end{proposition}
\begin{proof}

By Theorem \ref{Mis.ss4}, it can be checked that $\cH(P_3)$ is miserable with
\begin{align*}
V_{0,1} & = \{(0,0,0), (1,0,1)\} , \\
V_{1,0} & = \{(0,0,1), (0,1,0), (1,0,1)\}.
\end{align*}
Moreover, $0,0$-positions form the set $\{(a,0,a) \mid a \in \Z^+, a \geq 2\}$.

By Theorem \ref{Dos.ss4} it can be checked that $\cH(P_4)$ is domestic with
\begin{align*}
V_{0,0} & = \{[(a,b,0,a+b)] \mid a, b \in \Z^+, a + b \geq 2\}, \\
V_{0,1} & = \{(0,0,0,0), (0,1,0,1), (1,0,1,0), (1,0,0,1)\} , \\
V_{1,0} & = \{(0,0,0,1), (0,0,1,0), (0,1,0,0), (1,0,0,0), (1,1,1,1)\}.
\end{align*}
Yet, $\cH(P_4)$ is not tame, since $(1,1,1,2)$ is a $(5,1)$-position

Similarly, $\cH(P_5)$ is domestic with
\begin{align*}
V_{0,0} & = [(a,b,c,0,a+b+c)] \cup [(a,b,0,d,e) \cup [(f,0,h,f,h)] \\
        & \text{ with } a+b = d+e, f < h, \\
V_{0,1} & = \{(0,0,0,0,0), [(1,0,1,0,0)]\} , \\
V_{1,0} & = \{(0,0,0,0,1),  (0,1,1,1,1)\}.
\end{align*}
Yet,  $\cH(P_5)$ is not tame, since $(1,1,1,2,0)$ is a $(5,1)$-position.

Finally, $\cH(P_6)$ is not domestic, since $(1, 0, 1, 1, 1, 2)$ is a $(4,0)$-position.
\end{proof}

Based on our calculations, we conjecture that
the family of domestic but not tame games is large;
for example, it contains the next two subfamilies.


\begin{definition} \label{Dom.Conj2}
Given a ground set  $V$  partitioned by four pairwise disjoint subsets $V_1, V_2, V_3, V_4$, let us set
$H_1 = V_1 \cup V_4, 
H_2 = V_2 \cup V_4, 
H_3 = V_3 \cup V_4, 
H_4 = V_1 \cup V_2 \cup V_3$, 
and $\cH = \{H_1, H_2, H_3, H_4\}$.
We denote the corresponding position by
$(|V_1|, |V_2|, |V_3|, |V_4|)$. \qed
\end{definition}

The game in Definition \ref{Dom.Conj2} is not tame: $(1,2,2,2)$ is a $(7,1)$-position.

\begin{definition} \label{Dom.Conj1}
Given a ground set  $V$  partitioned by five pairwise disjoint subsets $V_1, V_2, V_3, V_4$, let us set
$H_1 = V_1 \cup V_2 \cup V_5, 
H_2 =  V_3 \cup V_4 \cup V_5, 
H_3 =  V_1 \cup V_3 \cup V_5, 
H_4 =  V_2 \cup V_4$, 
and $\cH = \{H_1, H_2, H_3, H_4\}$.
We denote the corresponding position by
$(|V_1|, |V_2|, |V_3|, |V_4|, |V_5|)$. \qed
\end{definition}

The game in Definition \ref{Dom.Conj2} is not tame: $(1,1,1,1,1)$ is a $(1,5)$-position.

\section{Closing Remarks}
After mis\`{e}re play was considered by Grundy and Smith \cite{GS56} in 1956,
it is a commonplace that the SG theory for the  mis\`{e}re  play is much more  difficult
than for the normal play. The reason is as follows.
Although, by Remark \ref{normal-misere}, a simple transformation
of the digraph of a game allows to convert
the mis\`{e}re play in  $G$  to the normal play in  $G^-$
yet, a problem appears for the sums.
The mis\`{e}re play of a sum
$G^- = (G_1 + \cdots + G_n)^-$  differs
from the sum of the corresponding  mis\`{e}re games
$G' = G_1^- + \cdots + G_n^-$.
Indeed, by Remark \ref{normal-misere}, in the first case
we add one new terminal, and an extra move, to the whole sum, while
in the second case we add them to each game-summand.
Thus, in general, the SG functions  $\G$ and $\G^-$  may differ a lot.

The main goal of this paper is to outline cases when the above two functions are similar.
Although the SG theory is not directly applicable
to the  mis\`{e}re playing sums, in general, but
it is applicable, in case when each summand is pet, or miserable and forced, or
(a weaker requirement) tame and returnable.

This idea should be attributed to Bouton, who applied it
to {\sc Nim} as early as in 1901, long before the SG theory was developed.
The classical {\sc Nim} is the sum of  $n$  games, each of which
(the one-pile {\sc Nim}) is trivial. It is pet and forced.
For a pile of  $k$  tokens the normal SG function $\G(k) = k$, while
the  mis\`{e}re one  $\G^-(k) = k$  for  $k \geq 2$, but $\G^-(0) = 1$  and  $\G^-(1) = 0$.
Thus, there are only two swap positions:
$k=0$ is the $(0,1)$-position, and  $k=1$ is the $(1,0)$-position.
Each of them can be reached by one move from any non-swap, $(k,k)$,
position with  $k \geq 2$.

{\sc Nim} is the sum of  $n$  such games and it has similar properties.
Namely, $x = (x_1, \ldots, x_n)$  is a swap swap position of {\sc Nim} if and only if
$x_i$  is  $0$  or  $1$  for every  $i \in [n] = \{1, \ldots, n\}$.
Furthermore,  $x$  is a $(0,1)$-position when the number of ones in  $x$  is even, and
$x$  is a $(1,0)$-position when this number is odd.

Given a non-swap position $x = (x_1, \ldots, x_n)$, obviously, a
swap position can be reached from  $x$  by one move if and only if
$x_i > 1$  for exactly one  $i \in [n]$.
But in this case, obviously, there is a move from  $x$  to
a $(0,1)$-  as well as another move to a  $(1,0)$-position.
Thus, {\sc Nim} is miserable
(and hence, tame) but it is not pet.
In a pet game a $(0,1)$-  as well as a  $(1,0)$-position is
reachable in one move from every non-swap position.

Moreover, {\sc Nim} is forced, since after a swap position is reached,
the  $(0,1)$-  and  $(1,0)$-positions alternate in any play,
since the number of piles containing one token will decrease one by one.
From these observations Bouton concluded that
the normal and mis\`{e}re plays of {\sc Nim} are similar:
the winning moves, if any, coincide in each position, unless
a swap position can be reached by one move.
Only in such (critical) position the player
should inquire which version, normal or mis\`{e}re, is actually played,
and then make a move to the swap position of the corresponding parity.

In fact, the same properties hold whenever
each game-summand is tame (not necessarily pet or miserable) and
returnable (not necessarily forced).
Surprisingly many games have these properties.
Let us recall, for example, the game Euclid.
Its swap positions are the Fibonacci pairs
$(F_j, F_{j+1})$, which are $(0,1)$-  or  $(1,0)$-positions
if and only if  $j$  is even or odd, respectively.
There is only  one move from  $(F_j, F_{j+1})$  and  it leads to  $(F_j, F_{j-1})$.
Moreover, for every non-swap position either there is no move
to a swap one, or there is a move to an even Fibonacci pair,
as well as some other move to an odd one \cite{Gur07}.
Thus, the game Euclid is miserable and forced.

\smallskip

Every subtraction game is pet, as it was shown by Ferguson \cite{Fer74} in 1974;
all considered versions of Wythoff's games are miserable;
both are returnable but not forced; see Section \ref{S.App}.

\smallskip

Thus, the mis\`{e}re play of any (possibly, mixed) sum of the games mentioned above,
{\sc Nim}, Eucllid, or Wythoff,  is not more difficult than the normal play.

Let us note however that both may be difficult.
For example, no closed formula is known for the SG-function
of the standard Wythoff game or any of its versions considered in Section \ref{S.App},
but if such a formula, for the normal play, would be discovered,
it will immediately allow us to solve both the normal and mis\`{e}re play of a sum
that may include Wythoff-summands among others.

\smallskip

The sum is tame (resp., miserable, miserable and forced, miserable and returnable) whenever every summand is,
in which case  $\G^-$  is simply equal to  $\G$  in all positions but swap ones.
Thus, the winning player makes a move to a $(0,0)$-position
from every positions, except a critical one, in which case
(s)he makes a move to a $(0,1)$ position of the sum.

\medskip

At the end of 19th century students usually played
the mis\`{e}re version of {\sc Nim}, which was considered standard.
So, this game was the goal of Bouton.
Yet, a nicer formula, so called {\sc Nim-sum},
describes the SG function of the {\em normal} version.
For this  reason, Bouton solved it first and then noticed
that solution of the standard (that is, mis\`{e}re) version
can be easily obtained from it, since the game of {\sc Nim} is miserable and forced.
Thus, in \cite{Bou901} Bouton introduced,
for the special case of {\sc Nim}, five fundamental
concepts of game and graph theories that appears in general only much later:
(i) the $\P$-positions, or in other words,
the {\em kernel} of an acyclic directed graph,
(ii) the SG function, (iii) the mis\`{e}re play, (iv, v) miserable and forced games.


\end{document}